\documentclass[11pt]{article}
\usepackage{amsmath}
\usepackage{amsfonts}
\usepackage{amssymb}
\usepackage{mathrsfs}
\usepackage{amsthm}
\usepackage{palatino}
\usepackage[margin=2.5cm, vmargin={1.5cm}]{geometry}

\usepackage{dcolumn}

\usepackage{times}
\usepackage{graphicx}
\usepackage{xcolor}

\usepackage{pstricks}
\usepackage{pst-plot}

\DeclareMathOperator{\cl}{Cl_2}

\renewcommand\footnotemark{}

\begin{document}

\title{Rademacher's conjecture and expansions at roots of unity of products generating restricted partitions}

\author{
  Cormac ~O'Sullivan\footnote{{\it Date:} Jan 15, 2020.
\newline \indent \ \ \
  \newline \indent \ \ \
Support for this project was provided by a PSC-CUNY Award, jointly funded by The Professional Staff Congress and The City
\newline \indent \ \ \
University of New York.}
  }

\date{}

\maketitle

\def\s#1#2{\langle \,#1 , #2 \,\rangle}

\def\H{{\mathbf{H}}}
\def\F{{\frak F}}
\def\C{{\mathbb C}}
\def\R{{\mathbb R}}
\def\Z{{\mathbb Z}}
\def\Q{{\mathbb Q}}
\def\N{{\mathbb N}}
\def\G{{\Gamma}}
\def\GH{{\G \backslash \H}}
\def\g{{\gamma}}
\def\L{{\Lambda}}
\def\ee{{\varepsilon}}
\def\K{{\mathcal K}}
\def\Re{\mathrm{Re}}
\def\Im{\mathrm{Im}}
\def\PSL{\mathrm{PSL}}
\def\SL{\mathrm{SL}}
\def\Vol{\operatorname{Vol}}
\def\lqs{\leqslant}
\def\gqs{\geqslant}
\def\sgn{\operatorname{sgn}}
\def\res{\operatornamewithlimits{Res}}
\def\li{\operatorname{Li_2}}
\def\lip{\operatorname{Li}'_2}
\def\pl{\operatorname{Li}}

\def\clp{\operatorname{Cl}'_2}
\def\clpp{\operatorname{Cl}''_2}
\def\farey{\mathscr F}

\newcommand{\stira}[2]{{\genfrac{[}{]}{0pt}{}{#1}{#2}}}
\newcommand{\stirb}[2]{{\genfrac{\{}{\}}{0pt}{}{#1}{#2}}}
\newcommand{\norm}[1]{\left\lVert #1 \right\rVert}

\newcommand{\e}{\eqref}


\newtheorem{theorem}{Theorem}[section]
\newtheorem{lemma}[theorem]{Lemma}
\newtheorem{prop}[theorem]{Proposition}
\newtheorem{conj}[theorem]{Conjecture}
\newtheorem{cor}[theorem]{Corollary}
\newtheorem{assume}[theorem]{Assumptions}

\newcounter{coundef}
\newtheorem{adef}[coundef]{Definition}

\newcounter{counrem}
\newtheorem{remark}[counrem]{Remark}

\renewcommand{\labelenumi}{(\roman{enumi})}
\newcommand{\spr}[2]{\sideset{}{_{#2}^{-1}}{\textstyle \prod}({#1})}
\newcommand{\spn}[2]{\sideset{}{_{#2}}{\textstyle \prod}({#1})}

\numberwithin{equation}{section}

\let\originalleft\left
\let\originalright\right
\renewcommand{\left}{\mathopen{}\mathclose\bgroup\originalleft}
\renewcommand{\right}{\aftergroup\egroup\originalright}

\bibliographystyle{alpha}

\begin{abstract}
The generating function for  restricted partitions is a finite product with a Laurent expansion at each root of unity. The question of the behavior of these Laurent coefficients as the size of the product increases goes back to Rademacher and his work on partitions. Building on the methods of Drmota, Gerhold and previous results of the author, we complete this description and give the full asymptotic expansion of each coefficient at every root of unity. These techniques are also shown to give the asymptotics of Sylvester waves.
\end{abstract}

\section{Introduction}
\subsection{Restricted  partitions}
Set $(q)_N:=(1-q)(1-q^2) \cdots (1-q^N)$.
Then $1/(q)_N$ is a meromorphic function of $q\in \C$ with the well-known expansion about zero
\begin{equation}\label{ex0}
  \frac 1{(q)_N} = \sum_{n=0}^\infty p_N(n) \cdot q^n \qquad (|q|<1),
\end{equation}
where $p_N(n)$ indicates the number of partitions of $n$, restricted to have at most $N$ parts. The identity
\begin{equation} \label{-q}
  (1/q)_N = (-1)^N q^{-N(N+1)/2} (q)_N,
\end{equation}
 also implies the expansion at infinity
\begin{equation}\label{ex-if}
  \frac 1{(q)_N} = (-1)^N \sum_{n=0}^\infty p_N(n) \cdot (1/q)^{n+N(N+1)/2} \qquad (|q|>1).
\end{equation}

Since \e{ex0} and \e{ex-if} both diverge when $|q|=1$,  it is natural to also consider
the expansions of $1/(q)_N$  in  neighborhoods of roots of unity.
 The Laurent expansion about $\xi$, a primitive $k$th root of unity, may be written as
\begin{equation} \label{am}
  \frac 1{(q)_N} = \sum_{m=-\lfloor N/k\rfloor}^\infty A_m(\xi,N) \cdot (q-\xi)^m.
\end{equation}
For example, when $\xi=1$ and $N=3$ the sequence of coefficients begins
$$
\{A_m(1,3)\}_{m \gqs -3} = \{ -1/6, \ 1/4, \ -17/72, \ 25/144, \ -91/864, \ \dots\}.
$$
Later we will see  the formula
\begin{equation}\label{amsum}
 A_m(1,N)  = \frac{(-1)^N}{ N!} \sum_{j_0+j_1+j_2+ \cdots + j_N = N+m}
    B_{j_1}^{(m+2)}  B_{j_2}  \cdots  B_{j_N} \frac{1^{j_1} 2^{j_2} \cdots
N^{j_N}}{j_0 ! j_1 ! j_2 ! \cdots j_N!}
\end{equation}
as in Proposition \ref{kop}, where the notation $B_n$ in \e{amsum} indicates the $n$th Bernoulli number, and $B_n^{(\alpha)}$ its Norlund polynomial generalization with power series
\begin{equation} \label{nor}
  \left(\frac{z}{e^z-1} \right)^\alpha =\sum_{n=0}^\infty B_n^{(\alpha)} \frac{z^n}{n!} \qquad (|z|<2\pi),
\end{equation}
so that  $B_n = B_n^{(1)}$.  See \e{nor2} for a formula for $B_n^{(\alpha)}$.

We show in Section \ref{nume} that, in general, the coefficients $A_m(\xi,N)$ are in the field $\Q(\xi)$.  For fixed $m$ and $\xi$, our main goal is to understand the asymptotic behavior of $A_m(\xi,N)$ as $N\to \infty$.

\subsection{Unrestricted partitions and Rademacher's conjecture}
For $q\in \C$ we let
\begin{equation*}
  1/(q)_\infty := \lim_{N\to \infty} 1/(q)_N.
\end{equation*}
This limit only exists for $|q| \neq 1$, converging to two different holomorphic functions, as in \cite[p. 301]{Ra}:
\begin{equation}\label{ex0pp}
  \frac 1{(q)_\infty} = \begin{cases} \sum_{n=0}^\infty p(n) \cdot q^n \qquad & \text{if} \quad |q|<1, \\
0 \qquad & \text{if} \quad |q|>1.
\end{cases}
\end{equation}
The notation $p(n)$ denotes the number of unrestricted partitions of $n$. Therefore \e{ex0pp} extends  \e{ex0} and \e{ex-if} to the $N=\infty$ case. It would seem unlikely that the expansions at roots of unity, \e{am}, could be similarly extended. However, Rademacher succeeded in showing that the principal parts of \e{am} do indeed have  natural analogs when $N=\infty$.

To describe this we first recall Rademacher's famous convergent series for the partitions from \cite{Ra2}
in the form given in \cite[(128.1)]{Ra}:
\begin{equation}\label{parrad}
  p(n)=\frac{\pi^{5/2}}{12\sqrt{3}}\sum_{k=1}^\infty \frac{\mathcal A_k(n)}{k^{5/2}}L_{3/2}\left( \left(\frac{\pi}{12k}\right)^2 (24n-1)\right).
\end{equation}
The notation means
\begin{equation*}
  \mathcal A_k(n) := \sum_{\substack{0\leqslant h<k \\ (h,k)=1}} \omega_{hk} \cdot  e^{-2\pi i h n/k}, \quad \qquad L_{3/2}(y):=y^{-3/4}I_{3/2}(2\sqrt{y})
\end{equation*}
where $\omega_{hk}$ is a certain root of unity associated to the multiplier system of the Dedekind eta function and $I_{3/2}$ the usual $I$-Bessel function. Rademacher then inserted \e{parrad} into \e{ex0pp} (he was already investigating this possibility in the 1937 paper \cite{Ra2}), and after dexterous manipulations obtained the remarkable expansion
\begin{equation}\label{rraadd}
  \frac 1{(q)_\infty} = \sum_{\substack{0\leqslant h<k  \\ (h,k)=1}}
\sum_{\ell=1}^{\infty} \frac{C_{hk\ell}(\infty)}{(q-e^{2\pi ih/k})^\ell} \qquad(|q|<1)
\end{equation}
with Rademacher's coefficients  given by
\begin{equation*}
  C_{hk\ell}(\infty) := -\frac{1}{12\sqrt{3}} \left(\frac{\pi}{k}\right)^{5/2} \omega_{hk} \cdot e^{2\pi ih \ell /k}  \cdot  \Delta^{\ell -1}_\alpha L_{3/2}\left.\left( -\frac{\pi^2}{6 k^2} (\alpha+1)\right) \right|_{\alpha=1/24}.
\end{equation*}
Here $\Delta^{\ell -1}_\alpha$ is indicating $\ell-1$ applications of the difference operator which acts on any function as $\Delta_\alpha f(\alpha):=f(\alpha+1)-f(\alpha)$. For example
$$
C_{011}(\infty) = -\frac 6{25} -\frac{12 \sqrt{3}}{125 \pi}, \qquad C_{121}(\infty)  = \frac{\sqrt{3}-3}{25} +\frac{12 (\sqrt{3}+3)}{125 \pi}.
$$

Now \e{rraadd} invites a comparison with the partial fraction decomposition
\begin{equation}\label{tp}
\frac 1{(q)_N} =\sum_{\substack{0\leqslant h<k \leqslant N \\ (h,k)=1}}
\sum_{\ell=1}^{\lfloor N/k \rfloor} \frac{C_{hk\ell}(N)}{(q-e^{2\pi ih/k})^\ell}
\end{equation}
which is built out of all the principal parts of \e{am} with
\begin{equation*}
  C_{hk\ell}(N):=  A_{-\ell}(e^{2\pi i h/k},N). 
\end{equation*}
Perhaps inspired by the fact that $\lim_{N\to \infty}p_N(n)=p(n)$, Rademacher conjectured in \cite[p. 302]{Ra} that $\lim_{N\to \infty}C_{hk\ell}(N)=C_{hk\ell}(\infty)$.
A sequence of papers, including for example \cite{An,SZ,OS15}, examined this issue. Finally, independently and with different methods, Rademacher's conjecture was disproved in \cite{DrGe} and \cite{OS16,OS16b}.
\SpecialCoor
\psset{griddots=5,subgriddiv=0,gridlabels=0pt}
\psset{xunit=0.06cm, yunit=4.6cm}
\psset{linewidth=1pt}
\psset{dotsize=2pt 0,dotstyle=*}
\begin{figure}[ht]
\begin{center}
\begin{pspicture}(-20,2.8)(215,3.8) 

\savedata{\mydata}[
{{7, 3.7269}, {8, 3.5642}, {9,
  3.4411}, {10, 3.3491}, {11, 3.2812}, {12, 3.2317}, {13,
  3.1965}, {14, 3.1726}, {15, 3.1574}, {16, 3.1491}, {17,
  3.1461}, {18, 3.1471}, {19, 3.1512}, {20, 3.1576}, {21,
  3.1655}, {22, 3.1743}, {23, 3.1837}, {24, 3.1931}, {25,
  3.2022}, {26, 3.2109}, {27, 3.2188}, {28, 3.2258}, {29,
  3.2318}, {30, 3.2367}, {31, 3.2405}, {32, 3.2431}, {33,
  3.2446}, {34, 3.2450}, {35, 3.2444}, {36, 3.2428}, {37,
  3.2405}, {38, 3.2374}, {39, 3.2338}, {40, 3.2297}, {41,
  3.2254}, {42, 3.2209}, {43, 3.2164}, {44, 3.2120}, {45,
  3.2079}, {46, 3.2041}, {47, 3.2008}, {48, 3.1981}, {49,
  3.1960}, {50, 3.1945}, {51, 3.1938}, {52, 3.1937}, {53,
  3.1943}, {54, 3.1956}, {55, 3.1975}, {56, 3.2000}, {57,
  3.2029}, {58, 3.2061}, {59, 3.2097}, {60, 3.2134}, {61,
  3.2171}, {62, 3.2208}, {63, 3.2243}, {64, 3.2275}, {65,
  3.2304}, {66, 3.2327}, {67, 3.2346}, {68, 3.2358}, {69,
  3.2364}, {70, 3.2364}, {71, 3.2357}, {72, 3.2343}, {73,
  3.2324}, {74, 3.2300}, {75, 3.2271}, {76, 3.2238}, {77,
  3.2203}, {78, 3.2166}, {79, 3.2130}, {80, 3.2094}, {81,
  3.2060}, {82, 3.2029}, {83, 3.2003}, {84, 3.1982}, {85,
  3.1967}, {86, 3.1959}, {87, 3.1958}, {88, 3.1963}, {89,
  3.1976}, {90, 3.1995}, {91, 3.2021}, {92, 3.2051}, {93,
  3.2086}, {94, 3.2124}, {95, 3.2163}, {96, 3.2203}, {97,
  3.2241}, {98, 3.2277}, {99, 3.2309}, {100, 3.2335}, {101,
  3.2355}, {102, 3.2367}, {103, 3.2371}, {104, 3.2367}, {105,
  3.2354}, {106, 3.2332}, {107, 3.2303}, {108, 3.2266}, {109,
  3.2224}, {110, 3.2177}, {111, 3.2127}, {112, 3.2076}, {113,
  3.2026}, {114, 3.1980}, {115, 3.1938}, {116, 3.1904}, {117,
  3.1879}, {118, 3.1863}, {119, 3.1860}, {120, 3.1868}, {121,
  3.1890}, {122, 3.1924}, {123, 3.1970}, {124, 3.2026}, {125,
  3.2092}, {126, 3.2166}, {127, 3.2243}, {128, 3.2323}, {129,
  3.2402}, {130, 3.2476}, {131, 3.2542}, {132, 3.2598}, {133,
  3.2639}, {134, 3.2665}, {135, 3.2671}, {136, 3.2658}, {137,
  3.2624}, {138, 3.2569}, {139, 3.2493}, {140, 3.2399}, {141,
  3.2289}, {142, 3.2165}, {143, 3.2033}, {144, 3.1896}, {145,
  3.1760}, {146, 3.1630}, {147, 3.1512}, {148, 3.1411}, {149,
  3.1333}, {150, 3.1282}, {151, 3.1263}, {152, 3.1278}, {153,
  3.1329}, {154, 3.1417}, {155, 3.1542}, {156, 3.1700}, {157,
  3.1888}, {158, 3.2100}, {159, 3.2331}, {160, 3.2571}, {161,
  3.2812}, {162, 3.3044}, {163, 3.3258}, {164, 3.3442}, {165,
  3.3588}, {166, 3.3686}, {167, 3.3729}, {168, 3.3710}, {169,
  3.3626}, {170, 3.3475}, {171, 3.3257}, {172, 3.2976}, {173,
  3.2639}, {174, 3.2256}, {175, 3.1837}, {176, 3.1398}, {177,
  3.0954}, {178, 3.0524}, {179, 3.0128}, {180, 2.9783}, {181,
  2.9508}, {182, 2.9322}, {183, 2.9238}, {184, 2.9270}, {185,
  2.9425}, {186, 2.9708}, {187, 3.0119}, {188, 3.0651}, {189,
  3.1292}, {190, 3.2026}, {191, 3.2829}, {192, 3.3676}, {193,
  3.4533}, {194, 3.5367}, {195, 3.6141}, {196, 3.6817}, {197,
  3.7358}, {198, 3.7729}, {199, 3.7900}, {200, 3.7845}}
]
\dataplot[linecolor=black,linewidth=0.8pt,plotstyle=dots]{\mydata}


\psline[linecolor=gray]{->}(-10,2.85)(210,2.85)
\multirput(0,2.83)(50,0){5}{\psline[linecolor=gray](0,0)(0,0.04)}
\psline[linecolor=gray]{->}(0,2.8)(0,3.8)
\multirput(-1.5,2.9)(0,0.1){9}{\psline[linecolor=gray](0,0)(3,0)}

\psline[linecolor=red](0,3.216)(200,3.216)
\rput(214,2.9){$N$}
\rput(210,3.27){$C_{014}(\infty)$}

\rput(-9,3){$_{0.030}$}
\rput(-9,3.2){$_{0.032}$}
\rput(-9,3.4){$_{0.034}$}
\rput(-9,3.6){$_{0.036}$}

\rput(50,2.93){$_{50}$}
\rput(100,2.93){$_{100}$}
\rput(150,2.93){$_{150}$}
\rput(200,2.93){$_{200}$}

\end{pspicture}
\caption{$C_{014}(N)=A_{-4}(1,N)$ for $1 \leqslant N \leqslant 200$ \label{cfig}}
\end{center}
\end{figure}
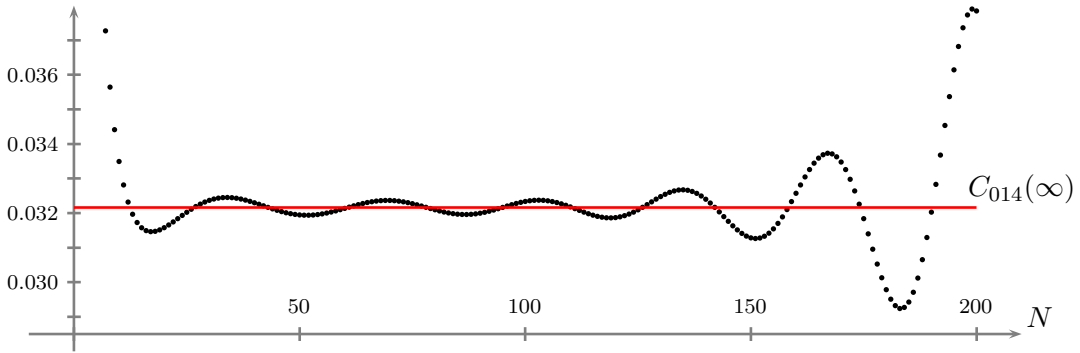
The problem is illustrated in Figure \ref{cfig}
where it may be seen that $C_{014}(N)$ is getting very close to $C_{014}(\infty)\approx 0.03216$ for $N$ up to about $100$. After that, though, the values diverge.

Drmota and Gerhold in \cite{DrGe} succeeded in giving the main term of the  asymptotics of $C_{01\ell}(N)$ as $N\to \infty$. On the other hand,  the complete asymptotic expansion of an average of Rademacher's coefficients $C_{hk\ell}(N)$ was given by the author in \cite{OS16,OS16b}. In this paper it is shown that  combining these two approaches with some further analysis allows us to give the complete asymptotic expansion of each $C_{hk\ell}(N)$ as $N\to \infty$. The result is that $C_{hk\ell}(N)$ always eventually oscillates like a sine wave in $N$ with period approximately $31.963 k$ and exponentially growing amplitude. An interesting number $w_0 \approx 0.916 - 0.182 i$ is controlling all this behavior; it is a zero of the analytically continued dilogarithm $\li(w)$ as described in Section \ref{moi}.

The observation that $C_{hk\ell}(N)$  gets very close to $C_{hk\ell}(\infty)$ for relatively small $N$, with this agreement seeming to improve as $\ell$ increases, deserves an explanation. See also \cite[Sect. 4]{SZ} and \cite[Table 2]{OS15}. We hope to return to this issue in a future work.

\subsection{Main results}
For the above dilogarithm zero $w_0$, set $z_0:= 2\pi i+\log(1-w_0)  \approx -1.606 + 7.423 i$. This is the saddle-point that appears in the analysis.  Since our results are valid for every coefficient, not just those in the principal part, we state  them using the notation $A_m(\xi,N)$ from \e{am}. The simplest case looks at the expansion about $\xi=1$:

\begin{theorem}\label{mainthm} Fix $m\in \Z$. For all $N \in \Z_{\gqs 1}$ we have
\begin{equation}
 \label{maineq}
   A_m(1,N)
   = \Re\bigg[\frac{w_0^{-N}}{N^{1-m}} \left( c_{m,0}+\frac{c_{m,1}}{N}+ \dots +\frac{c_{m,r-1}}{N^{r-1}}\right)\bigg] + O\left(\frac{|w_0|^{-N}}{N^{r+1-m}}\right)
\end{equation}
with an implied constant\footnote{For functions $f$ and $g$ we use the big O notation $f=O(g)$, or equivalently $f \ll g$, to indicate that there exists an implied constant $C$ so that $|f|\lqs C\cdot g$ for all values of the given variables in  specified ranges.}  depending only on $m$ and the positive integer $r$.
The main term has
\begin{equation*}
  c_{m,0} =  -1/(\pi i \cdot z_0^{m} \cdot  e^{ z_0/2}),
\end{equation*}
 and in general the numbers $c_{m,j}$ are given by the formula \e{c01ell}.
\end{theorem}

Note that
\begin{equation*}
  \Re\left[c \cdot w_0^{-N}\right] = |c|e^{U N}\sin(V N+\arg(ic))
\end{equation*}
for
\begin{equation}\label{uv}
  U :=-\log |w_0| \approx 0.0680762, \quad V :=\arg(1/w_0) \approx 0.196576,
\end{equation}
demonstrating the  oscillating and exponentially increasing behavior of $A_m(1,N)$. This behavior first appeared in 2011, in the extensive computations of Sills and Zeilberger  \cite{SZ}. Theorem \ref{mainthm} in the case $m=-1$ and $r=1$ was Conjecture 6.2 of \cite{OS15} (published online by the journal in 2012). Theorem \ref{mainthm} for negative $m$ and $r=1$, with a slightly weaker error term, was proved by Drmota and Gerhold as the main result of \cite{DrGe}.   The full Theorem \ref{mainthm} (at least for $m$ negative), with the same   coefficients $c_{m,j},$  was Conjecture 1.5 of \cite{OS16}. Table \ref{jeb} compares  the theorem for $m=-1$, $N=2500$ and different values of $r$, with the actual value of $A_m(1,N)$ computed using \e{wavek2}. All decimals are correct to the accuracy shown.

\begin{table}[ht]
\centering
\begin{tabular}{ccc}
\hline
 $r$   & Theorem \ref{mainthm} & \\
\hline
 $1$   & $3.8\textcolor{gray}{4650116057434743} \times 10^{67}$ & \\
 $3$    & $3.83861\textcolor{gray}{839292505665} \times 10^{67}$ & \\
 $5$   & $3.838617993\textcolor{gray}{36810779} \times 10^{67}$ & \\
 $7$    & $3.838617993486\textcolor{gray}{50473} \times 10^{67}$  &  \\
\hline
\phantom{$A_{-1}(1,2500)$} & $ 3.83861799348646318 \times 10^{67}$ & $A_{-1}(1,2500)$\\
\hline
\end{tabular}
\caption{The approximations of Theorem \ref{mainthm} to $A_{-1}(1,2500)$.} \label{jeb}
\end{table}

The next simplest case treats the expansion coefficients in \e{am} about $\xi=-1$.

\begin{theorem}\label{mainthm2} Fix $m\in \Z$. Let $N_2$ denote the residue class of $N \bmod 2$. Then for all $N \in \Z_{\gqs 1}$ we have
\begin{equation}
 \label{maineq2}
   A_m(-1,N)
   = \Re\left[\frac{w_0^{-N/2}}{N^{1-m}} \left( d_{m,0}(N_2)+\frac{d_{m,1}(N_2)}{N}+ \dots +\frac{d_{m,r-1}(N_2)}{N^{r-1}}\right)\right] + O\left(\frac{|w_0|^{-N/2}}{N^{r+1-m}}\right)
\end{equation}
with an implied constant  depending only on $m$ and the positive integer $r$.
The main term has
\begin{equation*}
  d_{m,0}(N_2) = -\frac{\sqrt{2}}{ \pi i} \left(\frac{-2}{z_0}\right)^{m} e^{-z_0/2}  \left(1+(-1)^{N_2} e^{ z_0/2} \right)^{1/2}
\end{equation*}
 and in general the numbers $d_{m,j}(N_2)$ are given by twice \e{buzz} with $\rho=-1$ and $k=2$.
\end{theorem}

 Cases of Theorem \ref{mainthm2} were given in Conjecture 6.3 of \cite{OS15} and Conjecture 6.4 of \cite{OS16}.
With \e{maineq2} we see that the values $A_m(-1,N)$ are roughly $\sqrt{A_m(1,N)}$ in size and behave slightly differently depending on whether $N$ is odd or even. For the following general result, set $N_k$ to be $N \bmod k$ with $0\lqs N_k \lqs k-1$.

\begin{theorem}\label{mainthm3x} Fix $m\in \Z$ and  $\xi$  a primitive $k$th root of unity. Then for all  $N\in \Z_{\gqs 1}$ we have
\begin{equation}
 \label{maineq3x}
   \frac{A_m(\xi,N)}{N^{m-1}}
   = w_0^{-N/k}\sum_{j=0}^{r-1}\frac{e_{m,j}(\xi,N_k)}{N^{j}}  + \overline{w_0^{-N/k}}\sum_{j=0}^{r-1}\frac{\overline{e_{m,j}(\overline{\xi},N_k)}}{N^{j}}
+
O\left(\frac{|w_0|^{-N/k}}{N^{r}}\right)
\end{equation}
for an implied constant  depending only on $k$, $m$ and the positive integer $r$.
For $\rho=\xi$ or $\overline{\xi}$, the  numbers $e_{m,j}(\rho,N_k)$ are given in \e{buzz} with the first being
\begin{multline} \label{cas2}
  e_{m,0}(\rho,N_k) =    \frac{-z_0}{2\pi i e^{z_0/2}}
  \frac{(w_0 /k)^{1/2}}{(1-e^{ z_0/k})^{1/2} } \times \prod_{j=1}^{k-1} \left( \frac{1-\rho^{-j}e^{ z_0/k}}{1-\rho^{-j}}\right)^{j/k-1/2}\\
\times  \rho^{-m} ( z_0/k)^{-m -1}
\times w_0^{N_k/k}\prod_{j=1}^{N_k} \left( 1-\rho^{j} e^{ z_0/k}\right)^{-1}.
\end{multline}
\end{theorem}
Theorems \ref{mainthm} and \ref{mainthm2} are the specializations of Theorem \ref{mainthm3x} to the cases $k=1$, $\xi=1$ and $k=2$, $\xi=-1$ respectively. Note that empty products are always taken to mean $1$.
\begin{table}[ht]
\centering
\begin{tabular}{ccc}
\hline
 $r$   & Theorem \ref{mainthm3x} & \\
\hline
 $1$   & $-1.\textcolor{gray}{659865606172377} \times 10^{14} + \textcolor{gray}{8.051322074007782} \times 10^{14}i$ & \\
 $3$   & $-1.7293\textcolor{gray}{81228591672} \times 10^{14} + 7.893\textcolor{gray}{645244567389}  \times 10^{14}i$ & \\
 $5$   & $-1.7293466\textcolor{gray}{49041497}\times 10^{14} + 7.893754\textcolor{gray}{752690905}  \times 10^{14}i$ & \\
 $7$   & $-1.729346669\textcolor{gray}{809078}\times 10^{14} + 7.8937545945\textcolor{gray}{13810}  \times 10^{14}i$  &  \\
\hline
\phantom{$A_{-2}$} & $ -1.729346669988476 \times 10^{14}+ 7.893754594541664 \times 10^{14}i$ & $A_{-2}(e^{2\pi i/3},2500)$\\
\hline
\end{tabular}
\caption{The approximations of Theorem \ref{mainthm3x} to $A_{-2}(e^{2\pi i/3},2500)$.} \label{jeb-k3}
\end{table}
\SpecialCoor
\psset{griddots=5,subgriddiv=0,gridlabels=0pt}
\psset{xunit=0.03cm, yunit=0.4cm}
\psset{linewidth=1pt}
\psset{dotsize=2pt 0,dotstyle=*}
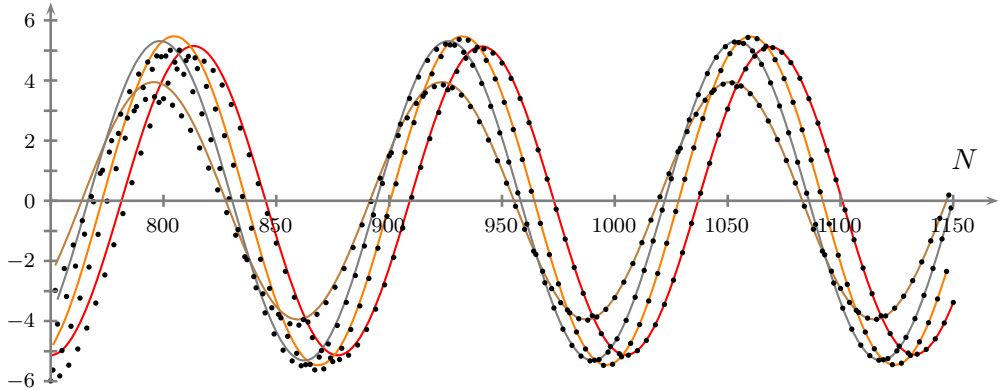
\begin{figure}[ht]
\begin{center}
\begin{pspicture}(720,-6.5)(1180,6.5) 

\savedata{\mydata}[
{{750, -5.1574}, {754, -5.0308}, {758, -4.7100}, {762, -4.2076},
{766, -3.5429}, {770, -2.7416}, {774, -1.8347}, {778, -0.8573}, {782,
  0.1530}, {786, 1.1571}, {790, 2.1163}, {794, 2.9936}, {798,
  3.7552}, {802, 4.3717}, {806, 4.8193}, {810, 5.0809}, {814,
  5.1463}, {818, 5.0132}, {822, 4.6866}, {826, 4.1793}, {830,
  3.5108}, {834, 2.7071}, {838, 1.7990}, {842,
  0.8218}, {846, -0.1870}, {850, -1.1883}, {854, -2.1436}, {858,
-3.0159}, {862, -3.7718}, {866, -4.3819}, {870, -4.8229}, {874,
-5.0778}, {878, -5.1367}, {882, -4.9974}, {886, -4.6654}, {890,
-4.1534}, {894, -3.4814}, {898, -2.6752}, {902, -1.7659}, {906,
-0.7887}, {910, 0.2187}, {914, 1.2176}, {918, 2.1693}, {922,
  3.0371}, {926, 3.7876}, {930, 4.3919}, {934, 4.8267}, {938,
  5.0753}, {942, 5.1280}, {946, 4.9830}, {950, 4.6458}, {954,
  4.1295}, {958, 3.4540}, {962, 2.6454}, {966, 1.7349}, {970,
  0.7576}, {974, -0.2487}, {978, -1.2453}, {982, -2.1937}, {986,
-3.0573}, {990, -3.8029}, {994, -4.4017}, {998, -4.8307}, {1002,
-5.0733}, {1006, -5.1203}, {1010, -4.9698}, {1014, -4.6276}, {1018,
-4.1071}, {1022, -3.4282}, {1026, -2.6173}, {1030, -1.7055}, {1034,
-0.7281}, {1038, 0.2772}, {1042, 1.2718}, {1046, 2.2171}, {1050,
  3.0768}, {1054, 3.8177}, {1058, 4.4114}, {1062, 4.8348}, {1066,
  5.0718}, {1070, 5.1132}, {1074, 4.9574}, {1078, 4.6105}, {1082,
  4.0859}, {1086, 3.4038}, {1090, 2.5905}, {1094, 1.6775}, {1098,
  0.6999}, {1102, -0.3046}, {1106, -1.2972}, {1110, -2.2396}, {1114,
-3.0956}, {1118, -3.8322}, {1122, -4.4209}, {1126, -4.8391}, {1130,
-5.0706}, {1134, -5.1067}, {1138, -4.9458}, {1142, -4.5943}, {1146,
-4.0658}, {1150, -3.3805}}
  ]
\dataplot[linecolor=red,linewidth=0.8pt,plotstyle=curve]{\mydata}

\savedata{\mydata}[
{{751, -4.7996}, {755, -4.1911}, {759, -3.4209}, {763, -2.5188},
{767, -1.5195}, {771, -0.4617}, {775, 0.6138}, {779, 1.6655}, {783,
  2.6529}, {787, 3.5378}, {791, 4.2862}, {795, 4.8691}, {799,
  5.2641}, {803, 5.4559}, {807, 5.4373}, {811, 5.2089}, {815,
  4.7795}, {819, 4.1658}, {823, 3.3914}, {827, 2.4862}, {831,
  1.4852}, {835,
  0.4270}, {839, -0.6477}, {843, -1.6972}, {847, -2.6812}, {851,
-3.5617}, {855, -4.3048}, {859, -4.8817}, {863, -5.2703}, {867,
-5.4556}, {871, -5.4304}, {875, -5.1958}, {879, -4.7607}, {883,
-4.1421}, {887, -3.3638}, {891, -2.4557}, {895, -1.4530}, {899,
-0.3944}, {903, 0.6795}, {907, 1.7270}, {911, 2.7079}, {915,
  3.5842}, {919, 4.3223}, {923, 4.8937}, {927, 5.2763}, {931,
  5.4553}, {935, 5.4240}, {939, 5.1835}, {943, 4.7431}, {947,
  4.1198}, {951, 3.3377}, {955, 2.4269}, {959, 1.4226}, {963,
  0.3635}, {967, -0.7096}, {971, -1.7553}, {975, -2.7332}, {979,
-3.6057}, {983, -4.3391}, {987, -4.9051}, {991, -5.2820}, {995,
-5.4551}, {999, -5.4179}, {1003, -5.1718}, {1007, -4.7263}, {1011,
-4.0986}, {1015, -3.3129}, {1019, -2.3994}, {1023, -1.3935}, {1027,
-0.3340}, {1031, 0.7385}, {1035, 1.7823}, {1039, 2.7575}, {1043,
  3.6262}, {1047, 4.3551}, {1051, 4.9161}, {1055, 5.2875}, {1059,
  5.4550}, {1063, 5.4122}, {1067, 5.1607}, {1071, 4.7103}, {1075,
  4.0783}, {1079, 3.2891}, {1083, 2.3731}, {1087, 1.3656}, {1091,
  0.3056}, {1095, -0.7662}, {1099, -1.8084}, {1103, -2.7808}, {1107,
-3.6460}, {1111, -4.3706}, {1115, -4.9267}, {1119, -5.2928}, {1123,
-5.4549}, {1127, -5.4066}, {1131, -5.1500}, {1135, -4.6948}, {1139,
-4.0586}, {1143, -3.2661}, {1147, -2.3476}}
  ]
\dataplot[linecolor=orange,linewidth=0.8pt,plotstyle=curve]{\mydata}

\savedata{\mydata}[
{{752, -2.1520}, {756, -1.4637}, {760, -0.7190}, {764, 0.0535}, {768,
  0.8239}, {772, 1.5626}, {776, 2.2409}, {780, 2.8328}, {784,
  3.3153}, {788, 3.6698}, {792, 3.8828}, {796, 3.9459}, {800,
  3.8567}, {804, 3.6186}, {808, 3.2409}, {812, 2.7382}, {816,
  2.1297}, {820, 1.4391}, {824,
  0.6929}, {828, -0.0800}, {832, -0.8498}, {836, -1.5869}, {840,
-2.2627}, {844, -2.8512}, {848, -3.3297}, {852, -3.6798}, {856,
-3.8878}, {860, -3.9459}, {864, -3.8517}, {868, -3.6089}, {872,
-3.2269}, {876, -2.7204}, {880, -2.1089}, {884, -1.4160}, {888,
-0.6685}, {892, 0.1048}, {896, 0.8740}, {900, 1.6096}, {904,
  2.2830}, {908, 2.8684}, {912, 3.3432}, {916, 3.6890}, {920,
  3.8924}, {924, 3.9458}, {928, 3.8469}, {932, 3.5997}, {936,
  3.2136}, {940, 2.7035}, {944, 2.0892}, {948, 1.3942}, {952,
  0.6455}, {956, -0.1281}, {960, -0.8968}, {964, -1.6309}, {968,
-2.3021}, {972, -2.8846}, {976, -3.3557}, {980, -3.6975}, {984,
-3.8966}, {988, -3.9455}, {992, -3.8422}, {996, -3.5908}, {1000,
-3.2008}, {1004, -2.6874}, {1008, -2.0704}, {1012, -1.3735}, {1016,
-0.6236}, {1020, 0.1503}, {1024, 0.9185}, {1028, 1.6512}, {1032,
  2.3202}, {1036, 2.8998}, {1040, 3.3676}, {1044, 3.7055}, {1048,
  3.9005}, {1052, 3.9451}, {1056, 3.8376}, {1060, 3.5821}, {1064,
  3.1885}, {1068, 2.6719}, {1072, 2.0523}, {1076, 1.3535}, {1080,
  0.6026}, {1084, -0.1716}, {1088, -0.9392}, {1092, -1.6706}, {1096,
-2.3375}, {1100, -2.9144}, {1104, -3.3789}, {1108, -3.7131}, {1112,
-3.9041}, {1116, -3.9446}, {1120, -3.8330}, {1124, -3.5737}, {1128,
-3.1765}, {1132, -2.6569}, {1136, -2.0348}, {1140, -1.3343}, {1144,
-0.5823}, {1148, 0.1921}}
  ]
\dataplot[linecolor=brown,linewidth=0.8pt,plotstyle=curve]{\mydata}

\savedata{\mydata}[
{{753, -3.2850}, {757, -2.4052}, {761, -1.4327}, {765, -0.4049}, {769,
   0.6385}, {773, 1.6572}, {777, 2.6119}, {781, 3.4659}, {785,
  4.1861}, {789, 4.7448}, {793, 5.1205}, {797, 5.2987}, {801,
  5.2724}, {805, 5.0428}, {809, 4.6187}, {813, 4.0164}, {817,
  3.2592}, {821, 2.3764}, {825, 1.4019}, {829,
  0.3733}, {833, -0.6696}, {837, -1.6866}, {841, -2.6386}, {845,
-3.4888}, {849, -4.2044}, {853, -4.7579}, {857, -5.1278}, {861,
-5.2999}, {865, -5.2677}, {869, -5.0323}, {873, -4.6029}, {877,
-3.9959}, {881, -3.2349}, {885, -2.3491}, {889, -1.3728}, {893,
-0.3435}, {897, 0.6990}, {901, 1.7145}, {905, 2.6638}, {909,
  3.5104}, {913, 4.2217}, {917, 4.7701}, {921, 5.1346}, {925,
  5.3011}, {929, 5.2632}, {933, 5.0223}, {937, 4.5878}, {941,
  3.9764}, {945, 3.2117}, {949, 2.3231}, {953, 1.3451}, {957,
  0.3151}, {961, -0.7269}, {965, -1.7410}, {969, -2.6878}, {973,
-3.5311}, {977, -4.2381}, {981, -4.7818}, {985, -5.1411}, {989,
-5.3021}, {993, -5.2588}, {997, -5.0127}, {1001, -4.5733}, {1005,
-3.9576}, {1009, -3.1894}, {1013, -2.2982}, {1017, -1.3185}, {1021,
-0.2879}, {1025, 0.7538}, {1029, 1.7664}, {1033, 2.7109}, {1037,
  3.5508}, {1041, 4.2539}, {1045, 4.7929}, {1049, 5.1472}, {1053,
  5.3030}, {1057, 5.2544}, {1061, 5.0033}, {1065, 4.5593}, {1069,
  3.9395}, {1073, 3.1679}, {1077, 2.2742}, {1081, 1.2928}, {1085,
  0.2616}, {1089, -0.7797}, {1093, -1.7909}, {1097, -2.7331}, {1101,
-3.5699}, {1105, -4.2690}, {1109, -4.8036}, {1113, -5.1531}, {1117,
-5.3038}, {1121, -5.2502}, {1125, -4.9942}, {1129, -4.5456}, {1133,
-3.9219}, {1137, -3.1470}, {1141, -2.2508}, {1145, -1.2678}, {1149,
-0.2360}}
  ]
\dataplot[linecolor=gray,linewidth=0.8pt,plotstyle=curve]{\mydata}

\savedata{\mydata}[
{{750, -5.9919}, {751, -5.6306}, {752, -2.9813}, {753, -4.1019},
{754, -5.8257}, {755, -4.9822}, {756, -2.2529}, {757, -3.1827}, {758,
-5.4680}, {759, -4.1747}, {760, -1.4703}, {761, -2.1732}, {762,
-4.9312}, {763, -3.2377}, {764, -0.6622}, {765, -1.1106}, {766,
-4.2344}, {767, -2.2057}, {768,
  0.1418}, {769, -0.0343}, {770, -3.4029}, {771, -1.1170}, {772,
  0.91220}, {773, 1.0155}, {774, -2.4675}, {775, -0.0122}, {776,
  1.6207}, {777, 1.9998}, {778, -1.4629}, {779, 1.0675}, {780,
  2.2414}, {781, 2.8820}, {782, -0.4266}, {783, 2.0816}, {784,
  2.7515}, {785, 3.6294}, {786, 0.6025}, {787, 2.9923}, {788,
  3.1326}, {789, 4.2141}, {790, 1.5861}, {791, 3.7655}, {792,
  3.3711}, {793, 4.6150}, {794, 2.4871}, {795, 4.3726}, {796,
  3.4590}, {797, 4.8176}, {798, 3.2720}, {799, 4.7913}, {800,
  3.3938}, {801, 4.8150}, {802, 3.9114}, {803, 5.0062}, {804,
  3.1791}, {805, 4.6085}, {806, 4.3817}, {807, 5.0102}, {808,
  2.8241}, {809, 4.2068}, {810, 4.6656}, {811, 4.8040}, {812,
  2.3434}, {813, 3.6264}, {814, 4.7530}, {815, 4.3963}, {816,
  1.7564}, {817, 2.8905}, {818, 4.6416}, {819, 3.8038}, {820,
  1.0866}, {821, 2.0283}, {822, 4.3363}, {823, 3.0501}, {824,
  0.3605}, {825, 1.0737}, {826, 3.8498}, {827,
  2.1650}, {828, -0.3931}, {829, 0.0644}, {830, 3.2015}, {831,
  1.1834}, {832, -1.1444}, {833, -0.9601}, {834, 2.4173}, {835,
  0.1437}, {836, -1.8637}, {837, -1.9596}, {838,
  1.5279}, {839, -0.9132}, {840, -2.5228}, {841, -2.8950}, {842,
  0.5684}, {843, -1.9460}, {844, -3.0954}, {845, -3.7295}, {846,
-0.4236}, {847, -2.9142}, {848, -3.5591}, {849, -4.4304},
{850, -1.4093}, {851, -3.7800}, {852, -3.8953}, {853, -4.9702},
{854, -2.3501}, {855, -4.5094}, {856, -4.0906}, {857, -5.3275}, {858,
-3.2092}, {859, -5.0739}, {860, -4.1369}, {861, -5.4881}, {862,
-3.9531}, {863, -5.4512}, {864, -4.0319}, {865, -5.4453}, {866,
-4.5525}, {867, -5.6262}, {868, -3.7794}, {869, -5.2004}, {870,
-4.9839}, {871, -5.5919}, {872, -3.3886}, {873, -4.7623}, {874,
-5.2304}, {875, -5.3491}, {876, -2.8741}, {877, -4.1476}, {878,
-5.2819}, {879, -4.9069}, {880, -2.2555}, {881, -3.3797}, {882,
-5.1363}, {883, -4.2819}, {884, -1.5562}, {885, -2.4877}, {886,
-4.7988}, {887, -3.4979}, {888, -0.8029}, {889, -1.5057}, {890,
-4.2821}, {891, -2.5849}, {892, -0.0243}, {893, -0.4713}, {894,
-3.6058}, {895, -1.5776}, {896, 0.7499}, {897,
  0.5759}, {898, -2.7959}, {899, -0.5148}, {900, 1.4900}, {901,
  1.5958}, {902, -1.8832}, {903, 0.5629}, {904, 2.1678}, {905,
  2.5495}, {906, -0.9027}, {907, 1.6142}, {908, 2.7574}, {909,
  3.4002}, {910, 0.1080}, {911, 2.5988}, {912, 3.2364}, {913,
  4.1156}, {914, 1.1102}, {915, 3.4790}, {916, 3.5864}, {917,
  4.6682}, {918, 2.0654}, {919, 4.2210}, {920, 3.7941}, {921,
  5.0370}, {922, 2.9370}, {923, 4.7964}, {924, 3.8519}, {925,
  5.2079}, {926, 3.6916}, {927, 5.1833}, {928, 3.7576}, {929,
  5.1745}, {930, 4.3003}, {931, 5.3670}, {932, 3.5150}, {933,
  4.9384}, {934, 4.7398}, {935, 5.3405}, {936, 3.1337}, {937,
  4.5087}, {938, 4.9935}, {939, 5.1050}, {940, 2.6286}, {941,
  3.9023}, {942, 5.0516}, {943, 4.6698}, {944, 2.0193}, {945,
  3.1427}, {946, 4.9121}, {947, 4.0519}, {948, 1.3294}, {949,
  2.2592},
{950, 4.5806}, {951, 3.2752}, {952, 0.5857}, {953, 1.2862}, {954,
  4.0701}, {955, 2.3698}, {956, -0.1830}, {957, 0.2613}, {958,
  3.4004}, {959, 1.3708}, {960, -0.9468}, {961, -0.7759}, {962,
  2.5974}, {963, 0.3168}, {964, -1.6763}, {965, -1.7853}, {966,
  1.6923}, {967, -0.7514}, {968, -2.3432}, {969, -2.7279}, {970,
  0.7201}, {971, -1.7925}, {972, -2.9217}, {973, -3.5671}, {974,
-0.2815}, {975, -2.7662}, {976, -3.3892}, {977, -4.2706}, {978,
-1.2739}, {979, -3.6350}, {980, -3.7278}, {981, -4.8110}, {982,
-2.2187}, {983, -4.3652}, {984, -3.9241}, {985, -5.1676}, {986,
-3.0793}, {987, -4.9286}, {988, -3.9707}, {989, -5.3264}, {990,
-3.8225}, {991, -5.3033}, {992, -3.8656}, {993, -5.2813}, {994,
-4.4195}, {995, -5.4749}, {996, -3.6127}, {997, -5.0339}, {998,
-4.8473}, {999, -5.4366}, {1000, -3.2217}, {1001, -4.5936}, {1002,
-5.0894}, {1003, -5.1900}, {1004, -2.7077}, {1005, -3.9775}, {1006,
-5.1363}, {1007, -4.7443}, {1008, -2.0904}, {1009, -3.2091}, {1010,
-4.9862}, {1011, -4.1168}, {1012, -1.3935}, {1013, -2.3180}, {1014,
-4.6448}, {1015, -3.3316}, {1016, -0.6439}, {1017, -1.3386}, {1018,
-4.1254}, {1019, -2.4190}, {1020,
  0.1296}, {1021, -0.3085}, {1022, -3.4478}, {1023, -1.4140}, {1024,
  0.8973}, {1025, 0.7326}, {1026, -2.6383}, {1027, -0.3554}, {1028,
  1.6295}, {1029, 1.7446}, {1030, -1.7279}, {1031, 0.7161}, {1032,
  2.2981}, {1033, 2.6886}, {1034, -0.7518}, {1035, 1.7592}, {1036,
  2.8773}, {1037, 3.5281}, {1038, 0.2524}, {1039, 2.7336}, {1040,
  3.3448}, {1041, 4.2310}, {1042, 1.2461}, {1043, 3.6020}, {1044,
  3.6828}, {1045, 4.7700}, {1046, 2.1909}, {1047, 4.3307}, {1048,
  3.8780}, {1049, 5.1245},
{1050, 3.0505}, {1051, 4.8919}, {1052, 3.9231}, {1053,
  5.2809}, {1054, 3.7918}, {1055, 5.2638}, {1056, 3.8164}, {1057,
  5.2331}, {1058, 4.3862}, {1059, 5.4322}, {1060, 3.5619}, {1061,
  4.9830}, {1062, 4.8108}, {1063, 5.3906}, {1064, 3.1695}, {1065,
  4.5403}, {1066, 5.0493}, {1067, 5.1407}, {1068, 2.6544}, {1069,
  3.9221}, {1070, 5.0926}, {1071, 4.6921}, {1072, 2.0365}, {1073,
  3.1523}, {1074, 4.9390}, {1075, 4.0621}, {1076, 1.3396}, {1077,
  2.2604}, {1078, 4.5946}, {1079, 3.2751}, {1080, 0.5906}, {1081,
  1.2810}, {1082, 4.0726}, {1083, 2.3615}, {1084, -0.1816}, {1085,
  0.2519}, {1086, 3.3933}, {1087,
  1.3564}, {1088, -0.9472}, {1089, -0.7874}, {1090, 2.5827}, {1091,
  0.2987}, {1092, -1.6767}, {1093, -1.7966}, {1094,
  1.6723}, {1095, -0.7709}, {1096, -2.3419}, {1097, -2.7370}, {1098,
  0.6972}, {1099, -1.8110}, {1100, -2.9171}, {1101, -3.5721}, {1102,
-0.3050}, {1103, -2.7816}, {1104, -3.3802}, {1105, -4.2699}, {1106,
-1.2956}, {1107, -3.6452}, {1108, -3.7132}, {1109, -4.8034}, {1110,
-2.2365}, {1111, -4.3686}, {1112, -3.9034}, {1113, -5.1519}, {1114,
-3.0913}, {1115, -4.9238}, {1116, -3.9433}, {1117, -5.3021}, {1118,
-3.8271}, {1119, -5.2894}, {1120, -3.8314}, {1121, -5.2482}, {1122,
-4.4154}, {1123, -5.4513}, {1124, -3.5720}, {1125, -4.9922}, {1126,
-4.8337}, {1127, -5.4032}, {1128, -3.1751}, {1129, -4.5440}, {1130,
-5.0658}, {1131, -5.1471}, {1132, -2.6559}, {1133, -3.9208}, {1134,
-5.1027}, {1135, -4.6927}, {1136, -2.0345}, {1137, -3.1467}, {1138,
-4.9430}, {1139, -4.0576}, {1140, -1.3349}, {1141, -2.2514}, {1142,
-4.5930}, {1143, -3.2662}, {1144, -0.5839}, {1145, -1.2696}, {1146,
-4.0660}, {1147, -2.3491}, {1148,
  0.1894}, {1149, -0.2389}, {1150, -3.3825}}
]
\dataplot[linecolor=black,linewidth=0.8pt,plotstyle=dots]{\mydata}


\psline[linecolor=gray]{->}(750,0)(1170,0)
\multirput(750,-0.2)(50,0){9}{\psline[linecolor=gray](0,0)(0,0.4)}
\psline[linecolor=gray]{->}(750,-6)(750,6.6)
\multirput(747,-6)(0,1){13}{\psline[linecolor=gray](0,0)(6,0)}

\rput(1155,1.4){$N$}
\rput(800,-0.8){$_{800}$}
\rput(850,-0.8){$_{850}$}
\rput(900,-0.8){$_{900}$}
\rput(950,-0.8){$_{950}$}
\rput(1000,-0.8){$_{1000}$}
\rput(1050,-0.8){$_{1050}$}
\rput(1100,-0.8){$_{1100}$}
\rput(1150,-0.8){$_{1150}$}

\rput(737,-6){$_{-6}$}
\rput(737,-4){$_{-4}$}
\rput(737,-2){$_{-2}$}
\rput(741,0){$_{0}$}
\rput(741,2){$_{2}$}
\rput(741,4){$_{4}$}
\rput(741,6){$_{6}$}

\end{pspicture}
\caption{$\Re(A_{-2}(i,N))*N^3 |w_0|^{N/4}$ for $750 \leqslant N \leqslant 1150$ \label{qfig}}
\end{center}
\end{figure}
Table \ref{jeb-k3} gives an example of the asymptotics of Theorem \ref{mainthm3x}, showing the increasing accuracy for larger $r$. Figure \ref{qfig} compares the real part of $A_{-2}(i,N)$ (shown as dots) with the real part of the main term ($r=1$) of Theorem \ref{mainthm3x}. We have divided by the increasing factor so that the main term of the theorem corresponds to four sine waves, depending on $N \bmod 4$. The figure shows excellent agreement after about $N=1000$.

Without much alteration, our techniques also give the asymptotics of Sylvester waves. Recall that Cayley  and Sylvester expressed the restricted partition function $p_N(n)$ in terms of polynomials. In Sylvester's formulation
\begin{equation} \label{sylthm}
    p_N(n)=\sum_{k=1}^N W_k(N,n)
\end{equation}
where each {\em wave} $W_k(N,n)$ is a polynomial in $n$ with rational coefficients and degree at most $\lfloor N/k \rfloor -1$. Note that $k$  polynomials are needed to represent $W_k(N,n)$, depending on the residue class of $n \bmod k$.
For example, if $n \equiv 0 \bmod 60$ then
\begin{equation*}
  p_5(n)
    = \overbrace{\frac{30 n^4+900 n^3+9300 n^2+38250 n+50651}{86400}}^{W_1(5,n)}
+\overbrace{\frac{2n+15}{128}}^{W_2(5,n)} +\overbrace{\frac{2}{27}}^{W_3(5,n)} +\overbrace{\frac{1}{16}}^{W_4(5,n)} +\overbrace{\frac{4}{25}}^{W_5(5,n)}.
\end{equation*}
Sylvester's Theorem  gives the formula
\begin{equation}\label{wv}
  W_k(N,n)= -\sum_{\xi} \xi^{-n}\res_{z=0} \frac{e^{-nz}}{(\xi e^z)_N}
\end{equation}
with the sum over all primitive $k$th roots of unity $\xi$. See for example \cite[pp. 119-135]{Di} for the history of this result. Note that \e{wv} makes sense for all $n\in \C$; we allow $n\in \Z$ below.

If $N$ is fixed, then we will have $p_N(n) \sim W_1(N,n)$ as $n\to \infty$. However, if  $N \to \infty$ and $n$ grows at most linearly with $N$, then the first wave $W_1(N,n)$ exhibits similar oscillating behaviour to $A_m(1,N)$ and diverges from $p_N(n)$. This is seen in the next result, giving the asymptotics of all waves in this regime.

\begin{theorem} \label{ma2}
Fix $k \in \Z_{\gqs 1}$ and $\lambda'  \in \R_{>0}$.  Then for all $N \in \Z_{\gqs 1}$ and $n \in \Z$ such that  $\lambda:= n/N$ satisfies $|\lambda| \lqs \lambda'$, we have
\begin{equation} \label{pres}
    W_k(N,n) = \Re\bigg[w_0^{-N/k}\sum_{j=0}^{r-1}\frac{a_{\lambda,j}(N_k,n_k)}{N^{j+2}}
\bigg] + O\left(\frac{|w_0|^{-N/k}}{N^{r+2}}\right)
\end{equation}
 where  $N_k \equiv N \bmod k$ and  $n_k \equiv n \bmod k$. The numbers $a_{\lambda,j}(N_k,n_k)$ are given explicitly in \e{wbuzz}. The implied constant depends only on $r \in \Z_{\gqs 1}$, $\lambda'$ and $k$.
\end{theorem}

In particular, the asymptotic expansion of the first wave is given by
\begin{equation*}
  W_1(N,\lambda N) = \Re\bigg[w_0^{-N}\sum_{j=0}^{r-1}\frac{a_{\lambda,j}(0,0)}{N^{j+2}}
\bigg] + O\left(\frac{|w_0|^{-N}}{N^{r+2}}\right)
\end{equation*}
with first coefficient $a_{\lambda,0}(0,0) = z_0 e^{-z_0(\lambda+1/2)}/(\pi i)$ as seen in \e{wcas2}. This confirms Conjecture 9.1 of \cite{OS18b}. The asymptotic expansion of the second wave is given by Theorem \ref{ma2} when $k=2$,  showing that Conjecture 9.2 of \cite{OS18b} is almost correct; the power $(-1)^N$ there should be $(-1)^{\lambda N}$. Table \ref{jebwv} gives an example of the theorem for the asymptotics of the third wave with $\lambda=1$.

\begin{table}[ht]
\centering
\begin{tabular}{ccc}
\hline
 $r$   & Theorem \ref{ma2} & \\
\hline
 $1$   &  $2.\textcolor{gray}{1243578451143945} \times 10^{32}$ & \\
 $3$    & $2.258\textcolor{gray}{2305404772980} \times 10^{32}$ & \\
 $5$   &  $2.2581936\textcolor{gray}{490316896} \times 10^{32}$ & \\
 $7$    & $2.2581936758\textcolor{gray}{669504} \times 10^{32}$  &  \\
\hline
\phantom{$W_{3}(4001,4001)$} & $2.2581936758249785 \times 10^{32}$ & $W_{3}(4001,4001)$\\
\hline
\end{tabular}
\caption{The approximations of Theorem \ref{ma2} to $W_{3}(4001,4001)$.} \label{jebwv}
\end{table}


\SpecialCoor
\psset{griddots=5,subgriddiv=0,gridlabels=0pt}
\psset{xunit=0.02cm, yunit=0.12cm}
\psset{linewidth=1pt}
\psset{dotsize=2pt 0,dotstyle=*}

\begin{figure}[ht]
\begin{center}
\begin{pspicture}(1150,65)(1850,110) 

\psline[linecolor=gray]{->}(1180,70)(1820,70)
\multirput(1200,69)(100,0){7}{\psline[linecolor=gray](0,0)(0,2)}
\psline[linecolor=gray]{->}(1200,68)(1200,110)
\multirput(1193,70)(0,10){4}{\psline[linecolor=gray](0,0)(14,0)}

\savedata{\mydata}[
{{1200., 69.7191}, {1202., 70.5394}, {1204., 70.9683}, {1206.,
  71.1997}, {1208., 71.2721}, {1210., 71.163}, {1212.,
  70.7402}, {1214., 68.6126}, {1216., 70.7919}, {1218.,
  71.6061}, {1220., 72.0328}, {1222., 72.2627}, {1224.,
  72.3337}, {1226., 72.2226}, {1228., 71.7954}, {1230.,
  69.5901}, {1232., 71.8648}, {1234., 72.6732}, {1236.,
  73.0976}, {1238., 73.3261}, {1240., 73.3956}, {1242.,
  73.2826}, {1244., 72.8509}, {1246., 70.5599}, {1248.,
  72.938}, {1250., 73.7405}, {1252., 74.1628}, {1254.,
  74.3898}, {1256., 74.4578}, {1258., 74.3428}, {1260.,
  73.9066}, {1262., 71.5203}, {1264., 74.0115}, {1266.,
  74.8082}, {1268., 75.2282}, {1270., 75.4537}, {1272.,
  75.5204}, {1274., 75.4034}, {1276., 74.9625}, {1278.,
  72.4692}, {1280., 75.0851}, {1282., 75.8761}, {1284.,
  76.294}, {1286., 76.518}, {1288., 76.5832}, {1290.,
  76.4642}, {1292., 76.0187}, {1294., 73.4033}, {1296.,
  76.159}, {1298., 76.9444}, {1300., 77.3601}, {1302.,
  77.5826}, {1304., 77.6463}, {1306., 77.5253}, {1308.,
  77.0752}, {1310., 74.3182}, {1312., 77.233}, {1314.,
  78.0129}, {1316., 78.4264}, {1318., 78.6475}, {1320.,
  78.7098}, {1322., 78.5867}, {1324., 78.1318}, {1326.,
  75.207}, {1328., 78.3073}, {1330., 79.0817}, {1332.,
  79.4931}, {1334., 79.7126}, {1336., 79.7735}, {1338.,
  79.6484}, {1340., 79.1887}, {1342., 76.0582}, {1344.,
  79.3817}, {1346., 80.1507}, {1348., 80.56}, {1350.,
  80.7781}, {1352., 80.8374}, {1354., 80.7103}, {1356.,
  80.2457}, {1358., 76.851}, {1360., 80.4564}, {1362., 81.22}, {1364.,
   81.6271}, {1366., 81.8438}, {1368., 81.9017}, {1370.,
  81.7725}, {1372., 81.303}, {1374., 77.5406}, {1376.,
  81.5312}, {1378., 82.2896}, {1380., 82.6946}, {1382.,
  82.9097}, {1384., 82.9662}, {1386., 82.8349}, {1388.,
  82.3604}, {1390., 77.9966}, {1392., 82.6062}, {1394.,
  83.3594}, {1396., 83.7622}, {1398., 83.9759}, {1400.,
  84.0309}, {1402., 83.8975}, {1404., 83.418}, {1406.,
  77.2261}, {1408., 83.6813}, {1410., 84.4294}, {1412.,
  84.8302}, {1414., 85.0424}, {1416., 85.0959}, {1418.,
  84.9604}, {1420., 84.4758}, {1422., 79.7458}, {1424.,
  84.7567}, {1426., 85.4997}, {1428., 85.8984}, {1430.,
  86.1091}, {1432., 86.1611}, {1434., 86.0236}, {1436.,
  85.5338}, {1438., 81.6163}, {1440., 85.8322}, {1442.,
  86.5702}, {1444., 86.9668}, {1446., 87.176}, {1448.,
  87.2266}, {1450., 87.0869}, {1452., 86.5919}, {1454.,
  83.1232}, {1456., 86.9078}, {1458., 87.6409}, {1460.,
  88.0354}, {1462., 88.2432}, {1464., 88.2923}, {1466.,
  88.1505}, {1468., 87.6501}, {1470., 84.4949}, {1472.,
  87.9836}, {1474., 88.7118}, {1476., 89.1043}, {1478.,
  89.3106}, {1480., 89.3582}, {1482., 89.2143}, {1484.,
  88.7085}, {1486., 85.7955}, {1488., 89.0596}, {1490.,
  89.7829}, {1492., 90.1734}, {1494., 90.3782}, {1496.,
  90.4243}, {1498., 90.2782}, {1500., 89.7671}, {1502.,
  87.0521}, {1504., 90.1357}, {1506., 90.8543}, {1508.,
  91.2427}, {1510., 91.4461}, {1512., 91.4907}, {1514.,
  91.3424}, {1516., 90.8257}, {1518., 88.2789}, {1520.,
  91.212}, {1522., 91.9258}, {1524., 92.3122}, {1526.,
  92.5141}, {1528., 92.5572}, {1530., 92.4068}, {1532.,
  91.8845}, {1534., 89.4842}, {1536., 92.2884}, {1538.,
  92.9975}, {1540., 93.3819}, {1542., 93.5824}, {1544.,
  93.624}, {1546., 93.4714}, {1548., 92.9434}, {1550.,
  90.6731}, {1552., 93.3649}, {1554., 94.0695}, {1556.,
  94.4518}, {1558., 94.6508}, {1560., 94.691}, {1562.,
  94.5361}, {1564., 94.0024}, {1566., 91.8492}, {1568.,
  94.4416}, {1570., 95.1416}, {1572., 95.5219}, {1574.,
  95.7195}, {1576., 95.7581}, {1578., 95.6011}, {1580.,
  95.0615}, {1582., 93.0151}, {1584., 95.5184}, {1586.,
  96.2139}, {1588., 96.5922}, {1590., 96.7883}, {1592.,
  96.8255}, {1594., 96.6662}, {1596., 96.1207}, {1598.,
  94.1726}, {1600., 96.5953}, {1602., 97.2863}, {1604.,
  97.6627}, {1606., 97.8574}, {1608., 97.893}, {1610.,
  97.7315}, {1612., 97.18}, {1614., 95.3231}, {1616.,
  97.6724}, {1618., 98.359}, {1620., 98.7334}, {1622.,
  98.9266}, {1624., 98.9607}, {1626., 98.7969}, {1628.,
  98.2394}, {1630., 96.4677}, {1632., 98.7496}, {1634.,
  99.4318}, {1636., 99.8042}, {1638., 99.996}, {1640.,
  100.029}, {1642., 99.8625}, {1644., 99.2988}, {1646.,
  97.6073}, {1648., 99.8269}, {1650., 100.505}, {1652.,
  100.875}, {1654., 101.066}, {1656., 101.097}, {1658.,
  100.928}, {1660., 100.358}, {1662., 98.7424}, {1664.,
  100.904}, {1666., 101.578}, {1668., 101.946}, {1670.,
  102.135}, {1672., 102.165}, {1674., 101.994}, {1676.,
  101.418}, {1678., 99.8737}, {1680., 101.982}, {1682.,
  102.651}, {1684., 103.018}, {1686., 103.205}, {1688.,
  103.233}, {1690., 103.06}, {1692., 102.478}, {1694.,
  101.002}, {1696., 103.06}, {1698., 103.725}, {1700.,
  104.089}, {1702., 104.275}, {1704., 104.302}, {1706.,
  104.127}, {1708., 103.537}, {1710., 102.127}, {1712.,
  104.137}, {1714., 104.798}, {1716., 105.161}, {1718.,
  105.346}, {1720., 105.371}, {1722., 105.193}, {1724.,
  104.597}, {1726., 103.249}, {1728., 105.215}, {1730.,
  105.872}, {1732., 106.233}, {1734., 106.416}, {1736.,
  106.439}, {1738., 106.259}, {1740., 105.657}, {1742.,
  104.369}, {1744., 106.293}, {1746., 106.946}, {1748.,
  107.305}, {1750., 107.487}, {1752., 107.508}, {1754.,
  107.326}, {1756., 106.717}, {1758., 105.487}, {1760.,
  107.371}, {1762., 108.02}, {1764., 108.377}, {1766.,
  108.557}, {1768., 108.578}, {1770., 108.393}, {1772.,
  107.776}, {1774., 106.602}, {1776., 108.449}, {1778.,
  109.094}, {1780., 109.449}, {1782., 109.628}, {1784.,
  109.647}, {1786., 109.46}, {1788., 108.836}, {1790.,
  107.717}, {1792., 109.528}, {1794., 110.169}, {1796.,
  110.522}, {1798., 110.699}, {1800., 110.716}}
]
\dataplot[linecolor=orange,linewidth=0.8pt,plotstyle=line]{\mydata}

\savedata{\mydata}[
{{1200., 79.8192}, {1205., 80.}, {1210., 80.1804}, {1215.,
  80.3604}, {1220., 80.5402}, {1225., 80.7195}, {1230.,
  80.8985}, {1235., 81.0771}, {1240., 81.2554}, {1245.,
  81.4333}, {1250., 81.6109}, {1255., 81.7881}, {1260.,
  81.965}, {1265., 82.1416}, {1270., 82.3178}, {1275.,
  82.4936}, {1280., 82.6692}, {1285., 82.8444}, {1290.,
  83.0192}, {1295., 83.1937}, {1300., 83.3679}, {1305.,
  83.5418}, {1310., 83.7154}, {1315., 83.8886}, {1320.,
  84.0615}, {1325., 84.2341}, {1330., 84.4063}, {1335.,
  84.5783}, {1340., 84.7499}, {1345., 84.9212}, {1350.,
  85.0922}, {1355., 85.2629}, {1360., 85.4333}, {1365.,
  85.6034}, {1370., 85.7732}, {1375., 85.9427}, {1380.,
  86.1118}, {1385., 86.2807}, {1390., 86.4493}, {1395.,
  86.6176}, {1400., 86.7855}, {1405., 86.9532}, {1410.,
  87.1206}, {1415., 87.2877}, {1420., 87.4546}, {1425.,
  87.6211}, {1430., 87.7873}, {1435., 87.9533}, {1440.,
  88.119}, {1445., 88.2844}, {1450., 88.4495}, {1455.,
  88.6143}, {1460., 88.7789}, {1465., 88.9432}, {1470.,
  89.1072}, {1475., 89.2709}, {1480., 89.4344}, {1485.,
  89.5976}, {1490., 89.7605}, {1495., 89.9232}, {1500.,
  90.0856}, {1505., 90.2477}, {1510., 90.4096}, {1515.,
  90.5712}, {1520., 90.7325}, {1525., 90.8936}, {1530.,
  91.0544}, {1535., 91.215}, {1540., 91.3753}, {1545.,
  91.5354}, {1550., 91.6952}, {1555., 91.8547}, {1560.,
  92.014}, {1565., 92.1731}, {1570., 92.3319}, {1575.,
  92.4904}, {1580., 92.6487}, {1585., 92.8068}, {1590.,
  92.9646}, {1595., 93.1222}, {1600., 93.2795}, {1605.,
  93.4366}, {1610., 93.5935}, {1615., 93.7501}, {1620.,
  93.9065}, {1625., 94.0626}, {1630., 94.2185}, {1635.,
  94.3742}, {1640., 94.5296}, {1645., 94.6848}, {1650.,
  94.8398}, {1655., 94.9945}, {1660., 95.149}, {1665.,
  95.3033}, {1670., 95.4574}, {1675., 95.6112}, {1680.,
  95.7648}, {1685., 95.9182}, {1690., 96.0714}, {1695.,
  96.2243}, {1700., 96.377}, {1705., 96.5295}, {1710.,
  96.6818}, {1715., 96.8338}, {1720., 96.9857}, {1725.,
  97.1373}, {1730., 97.2887}, {1735., 97.4399}, {1740.,
  97.5909}, {1745., 97.7417}, {1750., 97.8922}, {1755.,
  98.0426}, {1760., 98.1927}, {1765., 98.3426}, {1770.,
  98.4924}, {1775., 98.6419}, {1780., 98.7912}, {1785.,
  98.9403}, {1790., 99.0892}, {1795., 99.2379}, {1800., 99.3864}}
]
\dataplot[linecolor=red,linewidth=0.8pt,plotstyle=curve]{\mydata}

\savedata{\mydata}[
{{1200, 79.8192}, {1208, 80.1084}, {1216,
     80.3964}, {1224, 80.6834}, {1232, 80.9701}, {1240,
     81.2558}, {1248, 81.5397}, {1256, 81.8229}, {1264,
     82.1066}, {1272, 82.3892}, {1280, 82.6687}, {1288,
     82.9476}, {1296, 83.2294}, {1304, 83.5099}, {1312,
     83.7834}, {1320, 84.0568}, {1328, 84.3397}, {1336,
     84.6204}, {1344, 84.8832}, {1352, 85.1474}, {1360,
     85.4397}, {1368, 85.727}, {1376, 85.9656}, {1384,
     86.2092}, {1392, 86.535}, {1400, 86.8464}, {1408,
     87.0213}, {1416, 87.2084}, {1424, 87.6412}, {1432,
     88.0206}, {1440, 88.0196}, {1448, 88.0119}, {1456,
     88.7989}, {1464, 89.337}, {1472, 88.8383}, {1480.0,
     86.9685}, {1488.0, 90.0975}, {1496, 90.9028}, {1504,
     88.2183}, {1512, 91.0224}, {1520, 91.6541}, {1528,
     92.7369}, {1536, 91.7522}, {1544, 93.4829}, {1552,
     93.4957}, {1560, 94.7467}, {1568, 94.2507}, {1576,
     95.7063}, {1584, 95.5237}, {1592, 96.8368}, {1600,
     96.4962}, {1608, 97.8691}, {1616, 97.635}, {1624,
     98.9573}, {1632, 98.6795}, {1640, 100.014}, {1648,
     99.7767}, {1656, 101.089}, {1664, 100.844}, {1672,
     102.153}, {1680, 101.929}, {1688, 103.224}, {1696,
     103.004}, {1704, 104.292}, {1712, 104.084}, {1720,
     105.361}, {1728, 105.162}, {1736, 106.43}, {1744,
     106.241}, {1752, 107.499}, {1760, 107.32}, {1768,
     108.569}, {1776, 108.399}, {1784, 109.638}, {1792,
     109.478}, {1800, 110.708}}
]
\dataplot[linecolor=black,linewidth=0.8pt,plotstyle=dots]{\mydata}

\rput(1200,67){$_{1200}$}
\rput(1300,67){$_{1300}$}
\rput(1400,67){$_{1400}$}
\rput(1500,67){$_{1500}$}
\rput(1600,67){$_{1600}$}
\rput(1700,67){$_{1700}$}
\rput(1800,67){$_{1800}$}

\rput(1810,73){$n$}

\rput(1180,80){$_{80}$}
\rput(1180,90){$_{90}$}
\rput(1180,100){$_{100}$}

\rput(1555,78){$\log \big|\Re(z_0 e^{-3 z_0/2} w_0^{-n}/(\pi i n^2))\big|$}
\rput(1750,93){$\log p(n)$}

\rput(1450,96){$\log |W_1(n,n)|$}

\end{pspicture}
\caption{A `phase transition' for the first wave $W_1(n,n)$ at $n \approx 1480$ \label{wvfig}}
\end{center}
\end{figure}
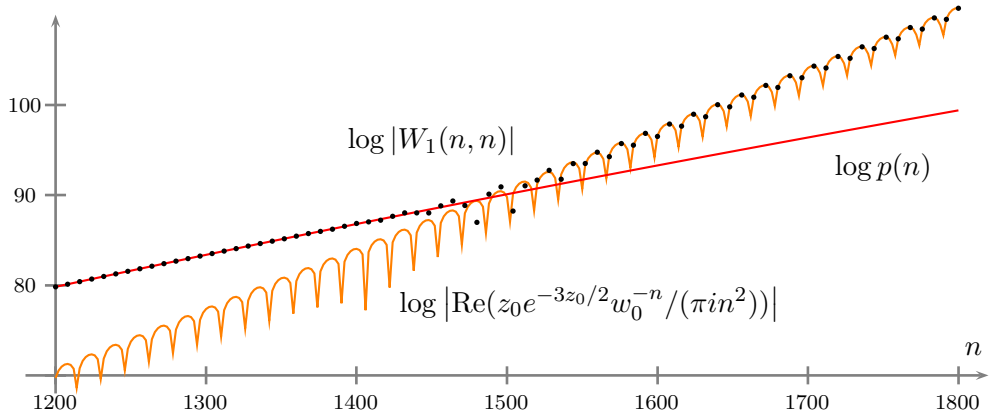


Also with $\lambda =1$, Figure \ref{wvfig} compares $W_1(n,n)$ with the first term of its asymptotic expansion, that is the real part of $z_0  e^{-3 z_0/2} w_0^{-n}/(\pi i n^2)$. In fact, $W_1(n,n)$ closely matches $p(n)$ for small $n$ and undergoes a kind of phase transition at about $n=1480$. For better visibility the figure displays the logs of the absolute values of these quantities. The values of $W_1(n,n)$ are shown as dots, sampled for $n \equiv 0 \bmod 8$. See Table 2 and Section 8.2 of \cite{OS18b} for further discussion of these comparisons. The second wave $W_2(n,n)$ exhibits similar behavior, staying close to a positive increasing function (that is approximately $p(n)^{1/2}/n^{2/3}$) until about $n=1600$ and after that following the expected oscillating asymptotics of Theorem \ref{ma2}. As with $C_{hk\ell}(N)$,  the interesting behaviour of $W_k(N,\lambda N)$ for small $N$ requires further investigation.

In this paper we focus on the product $1/(q)_N$. As $(q)_N$ and its $q$-Pochhammer variants are fundamental objects in $q$-series, we expect our techniques to be widely applicable to similar products, and series containing these products. For example, interesting evaluations  near roots of unity have already been made for the following  $q$-series:  quantum modular forms and Kashaev invariants of knots \cite{Za01,Hi}, Nahm sums \cite{Zag07,GZ}, mock theta functions \cite{FOR}, and $q$-binomial coefficients \cite{Zu}.

\subsection{Methods of proof}
Let $\xi$ be a primitive $k$th root of unity. Then by Cauchy's theorem
\begin{equation*}
  A_m(\xi,N)  =\res_{q=\xi} \frac 1{(q-\xi)^{m+1} (q)_N}
    = \frac 1{2\pi i} \int_{\mathcal C} \frac 1{(q-\xi)^{m+1} (q)_N} \, dq
\end{equation*}
where the closed path of integration $\mathcal C$ has a positive orientation, encircling $\xi$ and no other $N$th or lower roots of unity. With the change of variables $q=\xi e^{2\pi i \tau/N}$ we obtain
\begin{equation} \label{fary}
  A_m(\xi,N) = \frac 1N  \int_{\mathcal C'} \frac {\xi^{-m} e^{2\pi i \tau/N}}{(e^{2\pi i \tau/N}-1)^{m+1} (\xi e^{2\pi i \tau/N})_N} \, d\tau
\end{equation}
with $\mathcal C'$ enclosing $0$ and no other poles of $1/(\xi e^{2\pi i \tau/N})_N$. If $\xi=e^{2\pi i h/k}$ then these poles occur at
\begin{equation*}
  \tau= -\frac{Nh}k+\frac {Na}{b} \qquad \text{for} \qquad 1\lqs b \lqs N, \ a \in \Z.
\end{equation*}
If there is a pole at $\tau' \neq 0$  then
\begin{equation*}
  |\tau'| = \left|-\frac{Nh}k+\frac {Na}{b} \right| = \frac{N|-bh+a k|}{kb} \gqs \frac{N}{kb} \gqs \frac 1k
\end{equation*}
and so $\mathcal C'$ may be any circle centered at $0$ with radius less than $1/k$.

To reduce the notation,  a final change of variables  replaces $2\pi i \tau$ by $z$.
Therefore
\begin{equation} \label{qsz}
  A_m(\xi,N) = \frac 1N  \int_{\mathcal D} \frac {\xi^{-m} e^{z/N}}{2\pi i (e^{z/N}-1)^{m+1} (\xi e^{z/N})_N} \, dz
\end{equation}
for $\mathcal D$ a circle of radius less than $2\pi/k$.
Label the integrand in \e{qsz} as $Q_m(z;\xi,N)$ and let $\mathcal D'$ be the top half of $\mathcal D$, with imaginary part $\gqs 0$. Then $\overline{Q_m(z;\xi,N)} = -Q_m(\overline{z};\overline{\xi},N)$ implies that
\begin{equation} \label{qsz2}
  A_m(\xi,N) = \frac 1N  \int_{\mathcal D'} Q_m(z;\xi,N) \, dz + \frac 1N \overline{\int_{\mathcal D'} Q_m(z;\overline{\xi},N) \, dz}.
\end{equation}

This expression for $A_m(\xi,N)$ as a contour integral is our starting point. We will extend the methods of Drmota and Gerhold to write the integrands essentially in the form
\begin{equation*}
  e^{N \cdot p(z)}\left(\g_0(z)+\frac{\g_1(z)}{N}+\frac{\g_2(z)}{N^2}+ \cdots +\frac{\g_d(z)}{N^d} \right)
\end{equation*}
with $p(z)$ involving the dilogarithm.
This allows us to use the saddle-point method, in the precise form established by Perron, to achieve the desired asymptotics.

The following basic result, which will be needed in these arguments, is recorded here.

\begin{prop} \label{taylor}
Suppose $f(z)$ is holomorphic for $|z|<R$, equalling $\sum_{n=0}^\infty a_n z^n$ in this disk. Then for an implied constant depending only on $M \in \Z_{\gqs 0}$, $f$ and $R$
\begin{equation*}
  f(z)=\sum_{n=0}^{M-1} a_n z^n + O(|z|^M) \qquad \text{for} \qquad |z|\lqs R/2.
\end{equation*}
\end{prop}
\begin{proof}
Suppose $|z|\lqs R/2$. By Taylor's Theorem, as in for example \cite[pp. 125-126]{Al}, we have
\begin{equation} \label{amaz}
  f(z)-\sum_{n=0}^{M-1} a_n z^n = \frac{z^M}{2\pi i}\int_{|w|=3R/4} \frac{f(w)}{w^M(w-z)} \, dw.
\end{equation}
Let $C=\max |f(w)|$ for $w$ on the positively oriented circle of integration $|w|=3R/4$. Also note that $|w-z|\gqs R/4$. Then the absolute value of the right side of \e{amaz} is  $\lqs 3C(4/(3R))^M |z|^M$.
\end{proof}

\section{Estimating $1/(e^{z/N})_N$}

For the simplest case of $\xi=1$, we obtain from \e{qsz2}
\begin{equation} \label{qsz3}
  A_m(1,N) = \frac {2}N  \Re\bigg[\int_{\mathcal D'} \frac { e^{z/N}}{2\pi i (e^{z/N}-1)^{m+1} ( e^{z/N})_N} \, dz\bigg]
\end{equation}
for $\mathcal D'$ the top half of a circle of radius less than $2\pi$ and centered at $0$.
We will concentrate on this case first. The methods in this section follow the ideas in \cite[Sect. 2]{DrGe}, with improvements that allow us to obtain the complete asymptotic expansion of $1/( e^{z/N})_N$ as $N \to \infty$.

The {\em polylogarithm of order} $s \in \C$ is initially defined as
\begin{equation} \label{lis-def}
  \pl_s(z):=\sum_{n=1}^\infty \frac{z^n}{n^s} \qquad (|z|<1).
\end{equation}
Clearly \e{lis-def} also makes sense for $|z|\lqs 1$ when $\Re(s)>1$. The polylogarithm of order $2$ is called the {\em dilogarithm}.
Let $\delta_{i,j}$ denote the Kronecker delta function and set
\begin{equation} \label{pgdef}
p(z)  := \frac{\li(e^{ z})-\li(1)}z,\qquad
f_\ell(z)  :=\delta_{1,\ell} \frac z{24}+\frac{B_{\ell+1}}{(\ell+1)!} z^\ell \pl_{1-\ell}(e^z) \qquad (\ell \in \Z_{\gqs 1}).
\end{equation}

\begin{theorem} \label{coco}
Suppose $N\gqs 1$. For all $z\in \C$ with $\Re(z) \lqs -\delta <0$ and  $|z|\lqs C$ we have
\begin{equation*}
  \frac{1}{(e^{z/N})_N}  =
\left( \frac{- z}{2\pi N(1-e^{ z})}\right)^{1/2}
\exp\left(N \cdot p(z)+  \sum_{\ell =1}^{L-1} \frac{f_\ell(z)}{N^{\ell}}+O\left(N^{-L}\right)\right)
\end{equation*}
for an implied constant depending only on $L \in \Z_{\gqs 1}$, $\delta$ and $C$.
\end{theorem}

\begin{proof}
Beginning with
\begin{equation*}
  \log\left(1/(e^{z/N})_N \right) = - \log \prod_{j=1}^N (1- e^{j z/N}) = -  \sum_{j=1}^N \log(1- e^{ j z/N}),
\end{equation*}
it is clear that for $\Re(z)<0$
\begin{align}
  \log\left(1/(e^{z/N})_N \right)  & = \sum_{j=1}^N \sum_{r=1}^\infty \frac{e^{ j r z/N}}r  =  \sum_{r=1}^\infty \frac{1}r \sum_{j=1}^N  e^{ j r z/N}\notag\\
& =  \sum_{r=1}^\infty \frac{1}{r(e^{-r z/N} -1)}-\sum_{r=1}^\infty \frac{e^{r z}}{r(e^{-r z/N} -1)}. \label{aci}
\end{align}
Label the first series in \e{aci} as $\Phi(-z,N)$ and the second as $\chi(-z,N)$. In this notation
\begin{equation} \label{qqz}
  1/(e^{z/N})_N = \exp\left( \Phi(-z,N) - \chi(-z,N)\right).
\end{equation}

\begin{prop} \label{phi}
If $\Re(z)<0$, $N\gqs 1$ and $c>1$ then, for an implied constant depending only on $c$,
\begin{equation*}
  \Phi(-z,N) = -\frac{\pi^2 N}{6z}+ \frac z{24N}  +\frac 12 \log\left(\frac{- z}{2\pi N} \right) +O\left( \frac {|z|^c}{N^c} \right).
\end{equation*}
\end{prop}
\begin{proof}
We wish to estimate
\begin{equation*}
  \Phi(z,t)=\sum_{r=1}^\infty \frac 1{r(e^{rz/t}-1)} \qquad \text{for} \qquad x=\Re(z)>0, \ t>0.
\end{equation*}
Then
\begin{equation*}
  \sum_{r=1}^\infty \frac 1{r|e^{rz/t}-1|} \lqs \sum_{r=1}^\infty \frac 1{r(e^{r x/t}-1)}
\end{equation*}
and with
\begin{equation} \label{roo}
  \frac{u}{e^u-1} \lqs 1 \quad \text{for} \quad u> 0, \qquad \frac{1}{e^u-1} \lqs 2e^{-u} \quad \text{for} \quad u\gqs 1
\end{equation}
we find
\begin{equation*}
  \Phi(z,t)\ll \left\{
              \begin{array}{ll}
                e^{-x/t}, & \hbox{if $0<t\lqs x$;} \\
                t/x, & \hbox{if $x\lqs t < \infty$.}
              \end{array}
            \right.
\end{equation*}
Therefore the Mellin transform
\begin{equation*}
  \mathcal M\Phi(z,\cdot)(s):=\int_0^\infty \Phi(z,t)t^{s-1}\, dt
\end{equation*}
converges for $\Re(s)<-1$ and  interchanging integration and summation is possible:
\begin{align}
  \mathcal M\Phi(z,\cdot)(s)  & = \int_0^\infty \sum_{r=1}^\infty \frac 1{r(e^{rz/t}-1)} t^{s-1}\, dt \notag\\
   & = \sum_{r=1}^\infty \frac 1r \int_0^\infty  \frac{t^{s-1}}{e^{rz/t}-1} \, dt
 = \sum_{r=1}^\infty \frac 1r \int_0^\infty  \frac{u^{-s-1}}{e^{rz u}-1} \, du. \label{dday}
\end{align}
Since
\begin{equation*}
  \int_0^\infty \frac{u^{s-1}}{e^u-1}\, du = \G(s)\zeta(s) \qquad (\Re(s)>1)
\end{equation*}
we make the change of variables $w=rzu$ in \e{dday}. The new path of integration through $z$ may be moved to the positive real line and so, for $\Re(s)<-1$,
\begin{equation} \label{mell}
  \mathcal M\Phi(z,\cdot)(s) = z^s \sum_{r=1}^\infty r^{s-1} \int_0^\infty \frac{w^{-s-1}}{e^w-1}\, dw = z^s \G(-s)\zeta(-s)\zeta(1-s).
\end{equation}

The right side of \e{mell} has exponential decay as $|\Im(s)|\to \infty$ due to well-known bounds for $\G$ and $\zeta$. Therefore we may apply Mellin inversion, as in for example \cite[Thm. 2]{Fl95}, to obtain
\begin{align*}
  \Phi(z,t) & =\frac{1}{2\pi i} \int_{(c)} \mathcal M\Phi(z,\cdot)(s) \cdot  t^{-s}\, ds \\
& =\frac{1}{2\pi i} \int_{(c)} z^s \G(-s)\zeta(-s)\zeta(1-s) t^{-s}\, ds
\end{align*}
with integration along the vertical line with real part $c<-1$. The integrand only has poles at $s=-1, 0, 1$ as the trivial zeros of the zetas cancel the remaining  poles of the gamma function. Moving $c$ to the right and computing the residues shows
\begin{equation*}
   \Phi(z,t) = \frac{\pi^2 t}{6z}+\frac 12 \log \left(\frac z{2\pi t}\right)-\frac{z}{24 t} + \frac{1}{2\pi i} \int_{(c)} z^s \G(-s)\zeta(-s)\zeta(1-s) t^{-s}\, ds
\end{equation*}
for all $c>1$. This completes the proof of the proposition.
\end{proof}

The following easy lemma will be needed in the proof of Proposition \ref{chi}.

\begin{lemma} \label{crx}
For all real $c,d,x$ with $d,x >0$ we have
\begin{equation*}
  \sum_{r \in \Z, \ r\gqs d} r^c e^{-rx} \lqs
  \begin{cases}
     (1+1/x)   \cdot e^{-dx}, & \hbox{if $c\lqs 0$;} \\
(1+2/x) \cdot \left( c/x\right)^c  \cdot e^{-dx/2}, & \hbox{if $c>0$.}
  \end{cases}
\end{equation*}
\end{lemma}
\begin{proof}
If $c\lqs 0$ then $r^c\lqs 1$ and
\begin{equation*}
  \sum_{r \in \Z, \ r\gqs d} r^c e^{-rx} \lqs \sum_{r \in \Z, \ r\gqs d}  e^{-rx} \lqs \frac{e^{-dx}}{1-e^{-x}}.
\end{equation*}
With the left inequality in \e{roo} we have $(1-e^{-x})^{-1} \lqs 1+1/x$.

For $c>0$
write the summand as $r^c e^{-rx/2} \cdot e^{-rx/2}$. Then $r^c e^{-rx/2}$ is maximized for $r=2c/x$. Therefore
\begin{equation*}
  \sum_{r \in \Z, \ r\gqs d} r^c e^{-rx} \lqs \left( \frac{2c}{ex}\right)^c \sum_{r \in \Z, \ r\gqs d} e^{-rx/2}
\end{equation*}
and the result follows similarly.
\end{proof}

We next estimate $-\chi(-z,N)$. It is possible to do this using the Mellin transform approach of Proposition \ref{phi}, though this requires establishing bounds for the analytically continued polylogarithm. Instead  we use a more direct argument, taking advantage of the exponential decay of the numerator $e^{rz}$ in \e{aci}.

\begin{prop} \label{chi}
Suppose $N\gqs 1$, $\Re(z) \lqs -\delta <0$ and  $|z|\lqs C$. Then
\begin{equation*}
  -\chi(-z,N) = \frac N{ z}\li(e^{ z})  -\frac 12 \log(1-e^{ z}) +\sum_{\ell =2}^{L} 
\frac{B_{\ell}}{\ell!}  \left(\frac{ z}{N}\right)^{\ell-1}\pl_{2-\ell}(e^{ z}) +O\left( \frac {1}{N^{L}} \right)
\end{equation*}
for an implied constant depending only on $\delta$, $C$  and the positive integer $L$.
\end{prop}
\begin{proof}
We consider
\begin{equation} \label{czt}
  \chi(z,t) =\sum_{r=1}^\infty \frac{e^{-rz}}{r(e^{rz/t}-1)} \qquad \text{for} \qquad x=\Re(z)>0, \ t>0.
\end{equation}
Let $E_1$ be the tail of this series for $r\gqs t/|z|$. Employing the left inequality in \e{roo} shows the bound
\begin{align}
 |E_1| \lqs \sum_{r\gqs t/|z|} \frac{e^{-rx}}{r(e^{rx/t}-1)} & = \frac tx \sum_{r\gqs t/|z|} \frac{e^{-rx}}{r^2}
\frac{rx/t}{e^{rx/t}-1} \notag\\
   & \lqs \frac tx \sum_{r\gqs t/|z|} \frac{e^{-rx}}{r^2} \lqs \zeta(2)  \frac tx \exp\left(-\frac{t x}{|z|} \right). \label{erj}
\end{align}
By \e{nor} and Proposition \ref{taylor}
\begin{equation*}
  \frac{z}{e^z-1}=\sum_{\ell=0}^{L} B_\ell\frac{z^\ell}{\ell!} + O\left( |z|^{L+1}\right)
\end{equation*}
for $|z|\lqs \pi$, with an implied constant depending only on $L \gqs 1$. Then the initial part of the series
\e{czt} is
\begin{align}
  \sum_{r < t/|z|} \frac{e^{-rz}}{r(e^{rz/t}-1)} & = \sum_{r < t/|z|}
\frac{e^{-rz}}{r^2} \frac tz  \frac{rz/t}{e^{rz/t}-1} \notag\\
   & = \sum_{r < t/|z|}
\frac{e^{-rz}}{r^2} \frac tz  \left( \sum_{\ell=0}^{L} B_\ell\frac{(rz/t)^\ell}{\ell!} + O\left( |rz/t|^{L+1}\right)\right).
\label{az}
\end{align}
Hence \e{az} equals
\begin{equation}\label{az2}
  \sum_{\ell=0}^{L} \frac{B_\ell}{\ell!} \left( \frac zt \right)^{\ell-1}\sum_{r < t/|z|} r^{\ell-2}e^{-r z}
\end{equation}
with an error
\begin{equation}\label{errz}
  E_2 \ll \left( \frac{|z|}t \right)^{L} \sum_{r < t/|z|} r^{L-1}e^{-r x}
\ll \left( \frac{|z|}t \right)^{L}\left(1+\frac 1x\right) x^{1-L}
\end{equation}
using Lemma \ref{crx} with $d=1$.
Extending the sum over $r$ in \e{az2} to infinity produces the polylogarithm $\pl_{2-\ell}(e^{-z})$ and
 a further error
\begin{align}
 E_3 & \ll \sum_{\ell=0}^{L} \frac{|B_\ell|}{\ell!} \left( \frac{|z|}t \right)^{\ell-1}\sum_{r \gqs t/|z|} r^{\ell-2}e^{-r x}\notag \\
& \ll \left(\frac {t}{|z|}+\left(\frac{|z|}{t}\right)^{L-1} \right)
\sum_{\ell=0}^{L} \sum_{r \gqs t/|z|} r^{\ell-2}e^{-r x}\notag  \\
& \ll \left(\frac {t}{|z|}+\left(\frac{|z|}{t}\right)^{L-1} \right)
\left(1+\frac 1x\right) (1+x^{2-L})\exp\left(-\frac{t x}{2|z|} \right). \label{errz2}
\end{align}
We have shown
\begin{equation*}
  \chi(z,t) =  E_1+E_2+E_3 + \sum_{\ell=0}^{L} \frac{B_\ell}{\ell!} \left( \frac zt \right)^{\ell-1} \pl_{2-\ell}(e^{-z}).
\end{equation*}
For $0< \delta \lqs x$ and $|z|\lqs C$, \e{erj}, \e{errz} and \e{errz2} imply
\begin{equation} \label{erj3}
  E_1+E_2+E_3 \ll  \frac{1}{t^{L}} + (1+t+t^{1-L})  \exp\left(-\frac{\delta t}{2C} \right) \ll \frac{1}{t^{L}}
\end{equation}
for  an implied constant depending only on $L$, $\delta$ and $C$. Recall that $B_0=1$, $B_1=-1/2$ and  $\pl_1(z)=-\log(1-z)$. Hence
\begin{equation} \label{thyb}
  \chi(z,t) = \frac tz \li(e^{-z}) +\frac 12 \log(1-e^{-z}) + \sum_{\ell=2}^{L} \frac{B_{\ell}}{\ell!} \left( \frac zt \right)^{\ell-1} \pl_{2-\ell}(e^{-z}) +O\left(\frac 1{t^{L}}\right)
\end{equation}
and the proposition follows, where we simplified the signs using that  $B_\ell =0$ for $\ell \gqs 3$ and odd.
\end{proof}

Theorem \ref{coco} is an easy consequence of \e{qqz} and Propositions \ref{phi}, \ref{chi}.
\end{proof}

An interesting result with similarities to Theorem \ref{coco} is obtained in \cite[p. 53]{Zag07}.

\begin{cor} \label{cocor}
Suppose $N\gqs 1$ and $m \in \Z$. For all $z\in \C$ with $\Re(z) \lqs -\delta <0$ and  $|z|\lqs C$ we have
\begin{equation*}
  \frac { e^{z/N}}{ (e^{z/N}-1)^{m+1} ( e^{z/N})_N} \ll
N^{m+1}\exp\left(N \cdot \Re(p(z))\right)
\end{equation*}
for an implied constant depending only on $m$, $\delta$ and $C$.
\end{cor}
\begin{proof}
Suppose  $\Re(z) \lqs -\delta <0$ and  $|z|\lqs C$. By \e{nor} and Proposition \ref{taylor},
\begin{equation*}
  \left( \frac{z/N}{e^{z/N}-1}\right)^{m+1} = 1+O\left(\frac 1N\right)
\end{equation*}
when $N \gqs C/\pi$. 
Therefore
\begin{equation}\label{rty}
  \left| e^{z/N}-1 \right|^{-m-1} \ll N^{m+1}
\end{equation}
and, by increasing the implied constant if necessary, \e{rty} is valid for all $N\gqs 1$.

Clearly the numerator $e^{z/N}$  is  $O(1)$. Finally, the product $1/(e^{z/N})_N$ is bounded using Theorem \ref{coco} with $L=1$, and the factor $(1-e^z)^{-1/2}$ that appears is $O(1)$.
\end{proof}

The factor $\pl_{2-\ell}(e^{-z})$ in \e{thyb} may be expressed in different ways.
For $\Im(z)>0$ we have
\begin{equation*}
  \cot( z)=i-\frac{2i}{1-e^{2i z}} = -i -2i\sum_{j=1}^\infty e^{2j i z}
\end{equation*}
with $m$th derivative
\begin{equation*}
  \cot^{(m)}(z)=-\delta_{0,m} \cdot i -(2i)^{m+1} \pl_{-m}(e^{2i z}).
\end{equation*}
Therefore, in terms of the cotangent,
\begin{equation}  \label{cot}
  (2i)^{m+1} \pl_{-m}(e^{- z}) = -\delta_{0,m} \cdot i -\cot^{(m)}(i z/2) \qquad (m\in \Z_{\gqs 0}).
\end{equation}
For another formula,
recall the family of {\em Eulerian polynomials} beginning $ A_0(z)= A_1(z)=1$ and $A_2(z)=1+z$. In general, $A_m(z)$ has degree $m-1$ for $m\gqs 1$. Then, as in for example \cite[Sect. 8.2]{OS16c},
\begin{equation} \label{fro}
    \pl_{-m}(z) = \frac{z \cdot A_m(z)}{(1-z)^{m+1}} \qquad (m\in \Z_{\gqs 0}).
\end{equation}
The next lemma follows simply from the definition of $f_\ell(z)$ in \e{pgdef}, identity \e{fro} and \e{rty} with $N=1$.

\begin{lemma} \label{flem}
For all $z\in \C$ with $\Re(z) \lqs -\delta <0$ and  $|z|\lqs C$ we have
$f_\ell(z) = O(1)$
for an implied constant depending only on $\ell$, $\delta$ and $C$.
\end{lemma}

\section{The asymptotics of $A_m(1,N)$ as $N \to \infty$}

\subsection{Moving the path of integration} \label{moi}

Now we see from Theorem \ref{coco} that for large $N$ the main contribution to the integrand  in \e{qsz3} comes from the factor $\exp(N \cdot p(z))$, at least for $\Re(z)<0$. To apply the saddle-point method, the path of integration $\mathcal D'$ is moved to pass through a point $z_0$ where $p'(z_0)=0$, in such a way that $|\exp( p(z))|$ reaches its maximum only at $z=z_0$.

In solving the equation $p'(z)=0$, we may use the material in \cite[Sect. 2.3]{OS16}. In particular we need the case of \cite[Thm. 2.4]{OS16} with $d=0$ and $m=1$ (the variable $z$ in this theorem differs from ours by a factor of $2\pi i$). As shown there, employing one of the functional equations of the dilogarithm yields the identity
$$
 z^2 p'(z) = \li\left(1-e^{ z}\right) -2\pi i  \log \left(1-e^{ z}\right)
$$
for $\pi<\Im(z)<3\pi$. Let $w_0$ be a solution to $\li(w)-2\pi i  \log(w)=0$.
As reviewed in \cite[Sect. 2.3]{OS16},
$w_0$ is unique, may be computed to any accuracy with Newton's method, and is a zero of the dilogarithm on a non-principal branch. See \cite{OS16c} for further details.
Hence, according to the theorem, $z_0:= 2\pi i+\log(1-w_0)$ is a saddle-point of $p(z)$. To greater accuracy than the introduction,
\begin{equation}\label{zeros}
    w_0 \approx 0.91619782 - 0.18245890 i, \qquad z_0 \approx -1.6055276 + 7.4234262  i.
\end{equation}
It is also shown in \cite[Thm. 2.4]{OS16} that $p(z_0)=-\log(w_0)$.

\SpecialCoor
\psset{griddots=5,subgriddiv=0,gridlabels=0pt}
\psset{xunit=1.5cm, yunit=1.5cm, runit=0.8cm}
\psset{linewidth=1pt}
\psset{dotsize=3.5pt 0,dotstyle=*}
\psset{arrowscale=1.5,arrowinset=0.3}
\begin{figure}[ht]
\centering
\begin{pspicture}(-0.2,-1.3)(9.5,2.3) 

\psline[linecolor=gray](-0.2,-1)(0.4,-1)
\psline[linecolor=gray,linestyle=dotted,dotsep=1pt](0.4,-1)(0.9,-1)
\psline[linecolor=gray]{->}(0.9,-1)(3.5,-1)
\psline[linecolor=gray]{->}(2,-1.3)(2,2.3)

\multirput(1.95,-1)(0,0.5){7}{\psline[linecolor=gray](0,0)(0.1,0)}
\multirput(1.5,-1.05)(0.5,0){4}{\psline[linecolor=gray](0,0)(0,0.1)}
\psline[linecolor=gray](0,-1.05)(0,-0.95)

\psline[linecolor=gray]{->}(4.5,-1)(4.5,2)
\psline[linecolor=gray](4.5,0)(9.5,0)

\multirput(0,-0.05)(1,0){6}{\psline[linecolor=gray](4.5,0)(4.5,0.1)}
\multirput(4.45,-0.6)(0,0.2){11}{\psline[linecolor=gray](0,0)(0.1,0)}

\psline[linecolor=orange](3,-1)(3,-0.2)(1.95,-0.2)(1.784,0)
\psline(1.784,0)(1.676,0.5)
\psline[linecolor=orange](1.676,0.5)(1.352,2)(0.9,2)
\psline[linecolor=orange](0.4,2)(0,2)(0,-1)
\psline[linecolor=orange,linestyle=dotted,dotsep=1pt](0.4,2)(0.9,2)
\psline[linecolor=orange]{->}(0,0.5)(0,0.4)
\psline[linecolor=orange]{->}(1.568,1)(1.5464,1.1)
\psline[linecolor=orange]{->}(3,-0.6)(3,-0.5)

\psdot(1.74448, 0.181471)

\rput(0,-1.2){$_{-20\pi}$}
\rput(1.5,-1.2){$_{-\pi}$}
\rput(2.5,-1.2){$_{\pi}$}
\rput(3,-1.2){$_{2\pi}$}

\rput(2.27,0){$_{2\pi i}$}
\rput(2.27,1){$_{4\pi i}$}
\rput(2.27,2){$_{6\pi i}$}

\rput(3.3,-0.6){$L_1$}
\rput(2.5,-0.4){$L_2$}
\rput(1.7,-0.3){$L_3$}
\rput(1.4,0.2){$L_4$}
\rput(1.87,0.34){$z_0$}
\rput(1.2,1){$L_5$}
\rput(0.7,1.8){$L_6$}
\rput(-0.3,0.5){$L_7$}

\rput(4.89,-0.25){$L_3$}
\rput(6,-0.25){$L_4$}
\rput(7,-0.25){$L_5$}
\rput(8,-0.25){$L_6$}
\rput(9,-0.25){$L_7$}

\rput(4.2,1.4){$_{0.07}$}
\rput(4.2,1.0){$_{0.05}$}
\rput(4.2,0.6){$_{0.03}$}
\rput(4.2,0.2){$_{0.01}$}
\rput(4.13,-0.6){$_{-0.03}$}
\rput(4.13,-0.2){$_{-0.01}$}

\savedata{\mydata}[
{{4.9, -0.793941}, {4.95, -0.598411}, {5., -0.411226},
{5.05, -0.233264}, {5.1, -0.0652291}, {5.15, 0.0923735}, {5.2,
  0.239246}, {5.25, 0.375287}, {5.3, 0.500579}, {5.35,
  0.615352}, {5.4, 0.719962}, {5.45, 0.814854}, {5.5, 0.900539}}]
\dataplot[linecolor=red,linewidth=0.8pt,plotstyle=curve]{\mydata}

\psline[linecolor=blue,linestyle=dotted,dotsep=3pt](5.5,0)(5.5, 0.900539)

\savedata{\mydatab}[
{{5.5, 0.900539}, {5.55, 1.01973}, {5.6, 1.12305}, {5.65,
  1.20816}, {5.7, 1.27405}, {5.75, 1.32081}, {5.8, 1.34933}, {5.85,
  1.36103}, {5.9, 1.35764}, {5.95, 1.34106}, {6., 1.31321}, {6.05,
  1.27598}, {6.1, 1.23119}, {6.15, 1.18052}, {6.2, 1.12555}, {6.25,
  1.0677}, {6.3, 1.00829}, {6.35, 0.948486}, {6.4, 0.88933}, {6.45,
  0.831753}, {6.5, 0.776564}}]
\dataplot[linecolor=black,linewidth=0.8pt,plotstyle=curve]{\mydatab}

\psline[linecolor=blue,linestyle=dotted,dotsep=3pt](6.5,0)(6.5, 0.776564)
\psdot(5.85,0)
\rput(5.89,0.2){$z_0$}

\savedata{\mydatac}[
{{6.5, 0.776564}, {6.55, 0.631783}, {6.6, 0.527325}, {6.65,
  0.468061}, {6.7, 0.45083}, {6.75, 0.465317}, {6.8, 0.496526}, {6.85,
   0.528927}, {6.9, 0.5507}, {6.95, 0.555968}, {7., 0.544481}, {7.05,
  0.519814}, {7.1, 0.487343}, {7.15, 0.452657}, {7.2,
  0.420512}, {7.25, 0.394227}, {7.3, 0.375424}, {7.35,
  0.364056}, {7.4, 0.358714}, {7.45, 0.357142}, {7.5, 0.356875}}]
\dataplot[linecolor=red,linewidth=0.8pt,plotstyle=curve]{\mydatac}

\psline[linecolor=blue,linestyle=dotted,dotsep=3pt](7.5,0)(7.5, 0.356875)

\savedata{\mydatad}[
{{7.5, 0.356436}, {7.55, 0.569869}, {7.6, 0.720407}, {7.65,
  0.813125}, {7.7, 0.859472}, {7.75, 0.872655}, {7.8,
  0.864085}, {7.85, 0.842284}, {7.9, 0.813108}, {7.95, 0.78037}, {8.,
  0.746456}, {8.05, 0.712808}, {8.1, 0.680256}, {8.15,
  0.649245}, {8.2, 0.619983}, {8.25, 0.592532}, {8.3, 0.56687}, {8.35,
   0.542924}, {8.4, 0.520597}, {8.45, 0.499782}, {8.5, 0.480366}}]
\dataplot[linecolor=red,linewidth=0.8pt,plotstyle=curve]{\mydatad}

\psline[linecolor=blue,linestyle=dotted,dotsep=3pt](8.5,0)(8.5, 0.480366)

\savedata{\mydatae}[
{{8.5, 0.480366}, {8.55,
  0.484264}, {8.6, 0.488022}, {8.65, 0.491631}, {8.7,
  0.495082}, {8.75, 0.498369}, {8.8, 0.501483}, {8.85,
  0.504418}, {8.9, 0.507167}, {8.95, 0.509722}, {9., 0.512077}, {9.05,
   0.514227}, {9.1, 0.516166}, {9.15, 0.517889}, {9.2,
  0.519392}, {9.25, 0.52067}, {9.3, 0.521721}, {9.35, 0.522541}, {9.4,
   0.523128}, {9.45, 0.523481}, {9.5, 0.523599}}]
\dataplot[linecolor=red,linewidth=0.8pt,plotstyle=curve]{\mydatae}

\psline[linecolor=blue,linestyle=dotted,dotsep=3pt](9.5,0)(9.5, 0.523599)

\rput(7,1.7){$\Re(p(z))$}

\end{pspicture}
\caption{The path $\mathcal D''=L_1 \cup L_2 \cup \cdots \cup L_7$ and values of $\Re(p(z))$ on it}
\label{bfig}
\end{figure}
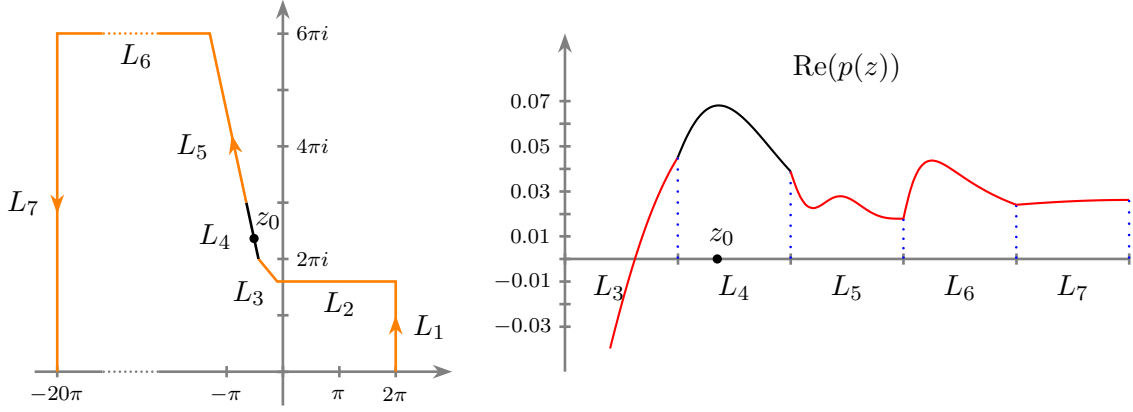

We move $\mathcal D'$ to the new path of integration $\mathcal D''$ as shown  on the left side of Figure \ref{bfig}, passing through $z_0$ and chosen so that the integrand stays  small away from $z_0$. A similar path is shown in \cite[Fig. 4]{DrGe}.
As no poles were crossed, \e{qsz3} implies
\begin{equation} \label{qaba}
    A_m(1,N) = \frac {2}N  \Re\bigg[\int_{\mathcal D''} \frac { e^{z/N}}{2\pi i (e^{z/N}-1)^{m+1} ( e^{z/N})_N} \, dz\bigg].
\end{equation}
Let $c:=\Re(z_0)/\Im(z_0)+i \approx -0.216+i$. Then the key part of this path is the line segment $L_4$ from $2\pi c$ to $3\pi c$ which contains $z_0$.
It is shown in the next section 
that the contribution from the rest of the path is $\ll e^{0.05 N}$ and so is negligible compared to the main asymptotics that turn out to be of size $\approx e^{0.068 N}$.

\subsection{Bounding the error} \label{bne}

As seen in Figure \ref{bfig}, the path of integration $\mathcal D'' \subset \C$ is made up of seven line segments. The first, $L_1$, goes vertically from $2\pi$ to $2\pi+8\pi i/5$. Next $L_2$ continues horizontally to $-\pi/10+8\pi i/5$. 
Then $L_3$ connects $-\pi/10+8\pi i/5$ to $2\pi c$, and as we saw, the segment $L_4$ is from $2\pi c$ to $3\pi c$. Lastly, $L_5$ continues to $6\pi c$, $L_6$ runs horizontally to $-20\pi+6\pi i$ and $L_7$ runs vertically down to $-20\pi$.

In this section we prove the following theorem, bounding  the integration in \e{qaba} outside of $L_4$.
\begin{theorem} \label{ban}
 Assume that  $N \gqs 1$. For an implied constant depending only on $m \in \Z$ we have
\begin{equation} \label{qb}
     \int_{\mathcal D''-L_4} \frac { e^{z/N}}{2\pi i (e^{z/N}-1)^{m+1} ( e^{z/N})_N} \, dz =  O\left( N^{m+1} e^{0.05 N}\right).
\end{equation}
\end{theorem}

\begin{proof}
After \e{qsz} we defined
\begin{equation}\label{qdef}
  Q_m(z;\xi,N) := \frac {\xi^{-m} e^{z/N}}{2\pi i (e^{z/N}-1)^{m+1} (\xi e^{z/N})_N}
\end{equation}
so that the integrand in \e{qb} is $ Q_m(z;1,N)$.
If $z\in L_3 \cup L_5 \cup L_6 \cup L_7$ then $\Re(z)<-\delta$ and $|z|<C$ for $\delta=\pi/10$ and $C=70$.
Hence, for these values of $z$, $\delta$ and $C$, Corollary \ref{cocor} shows that
\begin{equation*}
  Q_m(z;1,N) = O\left(N^{m+1} e^{N\cdot \Re(p(z))}\right)
\end{equation*}
for an implied constant depending only on $m\in \Z$.
It is straightforward to show that   $0.05$ is an upper bound for $\Re(p(z))$ when $z\in L_3 \cup L_5 \cup L_6 \cup L_7$; see the right side of Figure \ref{bfig}. We have therefore verified \e{qb} for integration along $L_3 \cup L_5 \cup L_6 \cup L_7$.

Note that \e{-q}
implies the symmetry
\begin{equation}\label{-z}
  Q_m(-z;1,N) = (-1)^{N+m+1}\exp((N+1)z/2+(m-1)z/N) \cdot Q_m(z;1,N).
\end{equation}
Suppose $z=x+i y\in L_1$ so that $x=2\pi$ and $0\lqs y \lqs 8\pi/5$. To bound $Q_m(z;1,N)$ we first apply \e{-z} and then Corollary \ref{cocor} as follows:
\begin{align*}
  |Q_m(z;1,N)| & = |\exp(-(N+1)z/2-(m-1)z/N)| \cdot |Q_m(-z;1,N)|  \\
& \ll \exp(-\pi(N+1)-2\pi(m+1)/N)  \cdot N^{m+1} \exp(N\cdot \Re(p(-z)))\\
   & \ll N^{m+1} \exp(N(-\pi+\Re(p(-z)))\\
& < N^{m+1}\exp(-2.8 N)
\end{align*}
since it may be shown that $\Re(p(-z))< 0.3$ when $z\in L_1$.

It only remains to verify \e{qb} when integrating over $L_2$.  Corollary \ref{cocor} cannot be used in this case because $L_2$ crosses the imaginary axis, so another way of estimating $Q_m(z;1,N)$ is found in the next proposition and corollary.
Recall the Clausen function $\cl(\theta):=\sum_{k=1}^\infty k^{-2}\sin(k \theta)=\Im(\li(e^{i \theta}))$, resembling a slanted sine wave.

\begin{prop} \label{jpp}
For $2\lqs N$ and $0<\theta<2\pi$   we have
\begin{equation} \label{jp}
  \prod_{j=1}^N \left| e^{ i j \theta/N}-1\right|^{-1} < \frac{2N}{\theta \sin(\theta/2)} \exp\left( N \frac {\cl(\theta)}{\theta}\right).
\end{equation}
\end{prop}
\begin{proof}
The identities
\begin{equation} \label{bojo}
  e^{i \theta}-1  = 2i e^{i \theta/2}(e^{i \theta/2} - e^{-i \theta/2})/(2i)=  2 \sin(\theta/2) \cdot e^{i(\theta+\pi)/2}
\end{equation}
show that
\begin{equation*}
  \log \prod_{j=1}^N \left| e^{ i j \theta/N}-1\right| = \log \prod_{j=1}^N \left( 2 \sin \frac{j \theta}{2N} \right) = \sum_{j=1}^N \psi(j)
\end{equation*}
for $0<\theta<2\pi$ and
\begin{equation*}
  \psi(t):=\log\left( 2 \sin \frac{t \theta}{2N} \right).
\end{equation*}
By Euler-Maclaurin summation
\begin{equation} \label{wwc}
  \sum_{j=1}^N \psi(j) = \int_1^N \psi(t)\, dt + \frac 12(\psi(N)+\psi(1)) + \int_1^N \left(t-\lfloor t \rfloor-\frac 12\right) \psi'(t)\, dt.
\end{equation}
The Clausen function has an alternate description as a log sine integral and this implies
\begin{equation*}
  \int_1^N \psi(t)\, dt = -\frac{N}\theta\left(\cl\left(\frac \theta{N} \cdot  N\right)-\cl\left(\frac \theta{N} \cdot 1 \right) \right).
\end{equation*}
Denote the last integral in \e{wwc} as $I$ so that
\begin{equation*}
  |I|  \lqs \frac 12 \int_1^N |\psi'(t)| \, dt.
\end{equation*}
Since $\psi'(t) =\frac{\theta}{2N}\cot(\frac{\theta t}{2N})$,
we have $\psi'(t) \gqs 0$ for $t\lqs \pi N/\theta$. It can only become negative if $\theta>\pi$.
Therefore
\begin{equation} \label{ex}
  \frac 12 \int_1^N |\psi'(t)| \, dt = \frac 12(\psi(N)-\psi(1)) \qquad \text{if $0<\theta\lqs \pi$}
\end{equation}
and if $\pi \lqs \theta<2\pi$
\begin{align}
  \frac 12 \int_1^N |\psi'(t)| \, dt
& = \frac 12 \int_1^{\pi N/\theta} \psi'(t) \, dt - \frac 12 \int_{\pi N/\theta}^N \psi'(t) \, dt \notag\\
& = \psi\left(\frac{\pi N}{\theta} \right) - \frac 12\left(\psi(N)+\psi(1) \right) \label{ex2}\\
& = \log 2 - \frac 12\left(\psi(N)+\psi(1) \right). \notag
\end{align}
Therefore
\begin{equation} \label{trb}
  -\log \prod_{j=1}^N \left| e^{ i j \theta/N}-1\right| \lqs
\frac{N}\theta\left(\cl\left(\theta \right)-\cl\left(\frac \theta{N}  \right) \right)
 -\psi(1)+ \left\{
             \begin{array}{ll}
               0, & \hbox{if $\theta\lqs \pi$;} \\
               \log 2 - \psi(N), & \hbox{if $\theta\gqs \pi$.}
             \end{array}
           \right.
.
\end{equation}
If $N\gqs 2$ then $\theta/N \lqs \pi$. Therefore $\cl(\theta/N)$ is positive and may be omitted from the bound \e{trb}. Also for $N\gqs 2$ we have $\sin(\theta/(2N))> \theta/(4N)$, say, so that $-\psi(1)<\log(2N/\theta)$. The proposition follows.
\end{proof}

The above proof also gives a lower bound for the product in \e{jp}. We will apply Proposition \ref{jpp} next with $\theta=8\pi/5$ which is close to the minimum of $\cl(\theta)/\theta$.

\begin{cor} \label{polo}
For $z=x+iy$ with $-\pi/10\lqs x \lqs 2\pi$ and $ y =8\pi/5$ we have
\begin{equation} \label{mag}
  Q_m(z;1,N) = O(N^{m+1})
\end{equation}
where $N\gqs 1$ and the  implied constant depends only on $m \in \Z$.
\end{cor}
\begin{proof}
For these $z$ values, a similar argument to the one for \e{rty} shows
\begin{equation} \label{siq}
  \frac { e^{z/N}}{2\pi i (e^{z/N}-1)^{m+1}} \ll N^{m+1}.
\end{equation}
Note that
\begin{equation*}
  \left| (e^{z/N})_N \right|^{-1} = \prod_{j=1}^N \left|1-e^{ j(x+i y)/N} \right|^{-1}
= \prod_{j=1}^N \left|e^{ i j y/N} -e^{ j x/N} \right|^{-1}.
\end{equation*}
Hence, for $x \gqs 0$,
\begin{equation*}
   Q_m(z;1,N) \ll  N^{m+1} \cdot \prod_{j=1}^N \left|e^{ i j y/N} -1 \right|^{-1}.
\end{equation*}
As $y=8\pi/5$, Proposition \ref{jpp} implies
\begin{equation}\label{dfd}
  \prod_{j=1}^N \left|e^{ i j y/N} -1 \right|^{-1} \ll N e^{-0.198 N}
\end{equation}
and we have shown \e{mag} for $0\lqs x \lqs 2\pi$.

The remaining case has $-\pi/10 \lqs x \lqs 0$ and
we use  identity \e{-q}
to find
\begin{equation*}
  \left| (e^{z/N})_N \right|^{-1} = e^{-(N+1)x/2}\left| (e^{-z/N})_N \right|^{-1}.
\end{equation*}
Combining \e{siq} with this implies
\begin{equation} \label{lab}
  Q_m(z;1,N) \ll  N^{m+1} \cdot e^{\pi N/20}\left| (e^{-z/N})_N \right|^{-1}.
\end{equation}
Then
\begin{equation*}
  \left| (e^{-z/N})_N \right|^{-1} = \prod_{j=1}^N \left|1-e^{ -j(x+i y)/N} \right|^{-1}
= \prod_{j=1}^N \left|e^{ i j y/N} -e^{ -j x/N} \right|^{-1} \lqs \prod_{j=1}^N \left|e^{ i j y/N} -1 \right|^{-1}.
\end{equation*}
Bounding this by \e{dfd} means that \e{lab} implies
\begin{equation*}
  Q_m(z;1,N) \ll   N^{m+1} \cdot N e^{\pi N/20} e^{-0.198 N}< N^{m+1} \cdot N e^{-0.03N} \ll N^{m+1}
\end{equation*}
as desired.
\end{proof}

The above proof shows  that the bound in \e{mag} may be improved to $O(N^{m+1}e^{-0.02N})$, say, though we do not require it. Similar estimates are made in \cite[Sect. 4]{DrGe} with
\begin{equation*}
  \prod_{j=1}^N \left|e^{ j z/N} -1 \right|^{-1} \ll  e^{-0.17 N} \quad \text{for} \quad
-N^{-7/8}\lqs x \lqs 0, \ y=5
\end{equation*}
shown for example in \cite[Eq. (22)]{DrGe}.
In any case, Corollary \ref{polo} gives a good enough bound on $Q_m(z;1,N)$ for $z\in L_2$ to  complete the proof of Theorem \ref{ban}.
\end{proof}

\subsection{The saddle-point method}
It follows from \e{qaba} and Theorem \ref{ban} that
\begin{equation} \label{jan}
    A_m(1,N) = \frac {2}N  \Re\bigg[\int_{L_4} \frac { e^{z/N}}{2\pi i (e^{z/N}-1)^{m+1} ( e^{z/N})_N} \, dz\bigg] +O( N^{m} e^{0.05 N}).
\end{equation}
The asymptotics of the integral in \e{jan} will be found by applying Perron's saddle-point method from \cite{Pe17}. The exact form we need is given in \cite{OSper} and requires the following discussion to state it precisely.
The usual convention that the principal branch of $\log$ has arguments in $(-\pi,\pi]$ is used.
As in  \eqref{hjw} below, powers of nonzero complex numbers take the corresponding principal value $z^{\tau}:=e^{\tau\log(z)}$
for $\tau \in \C$.


\begin{assume} \label{asma}
Let $\mathcal B$ be a neighborhood of $z_0 \in \C$.  Let $\mathcal L$ be  a closed, bounded contour such as the above line segment $L_4$, with $z_0$  a point on it. Suppose $p(z)$ and $q(z)$ are holomorphic functions on a domain containing $\mathcal B \cup \mathcal L$.  We  assume $p(z)$ is not constant and  hence there must exist $\mu \in \Z_{\gqs 1}$ and $p_0 \in \C_{\neq 0}$ so that
\begin{equation}
    p(z)  =p(z_0)-p_0(z-z_0)^\mu(1-\phi(z)) \qquad  (z\in \mathcal B) \label{funp}
\end{equation}
with $\phi$  holomorphic on $\mathcal B$ and $\phi(z_0)=0$.
Define the {\em steepest-descent angles}
\begin{equation}\label{bisec}
    \theta_r := -\frac{\arg(p_0)}{\mu}+\frac{2\pi r}{\mu} \qquad (r \in \Z).
\end{equation}
We also assume that  $\mathcal B,$ $\mathcal L,$ $p(z),$  $q(z)$ and  $z_0$  are independent of  $N>0$. Finally, let $K(q)$ be a bound for $|q(z)|$ on $\mathcal B \cup \mathcal L$.
\end{assume}

\begin{theorem}[The saddle-point method of Perron] \label{sdle}
Suppose Assumptions \ref{asma} hold with $\mu$ even and
\begin{equation} \label{<}
   \Re(p(z))<\Re(p(z_0)) \quad \text{for all} \quad z \in \mathcal L, \ z\neq z_0.
\end{equation}
Let $\mathcal L$ approach $z_0$ in a sector of angular width $2\pi/\mu$ about $z_0$ with bisecting angle $\theta_{r}\pm\pi$, and  initially leave $z_0$ in a sector of the same size with bisecting angle $\theta_{r}$.
 Then, for every $S \in \Z_{\gqs 0}$, there are explicit numbers $\alpha_{2s}(p,q;z_0)$ so that
\begin{equation*} 
    \int_\mathcal L e^{N \cdot p(z)} q(z) \, dz = 2 e^{N \cdot p(z_0)} \left(\sum_{s=0}^{S-1}  \G\left(\frac{2s+1}{\mu}\right) \frac{\alpha_{2s}(p,q;z_0)   \cdot e^{2\pi i r (2s+1)/\mu}}{N^{(2s+1)/\mu}} + O\left(\frac{K(q)}{N^{(2S+1)/\mu}} \right)  \right)
\end{equation*}
as $N \to \infty$. The implied constant  is independent of $N$ and $q$.
\end{theorem}

Theorem \ref{sdle} is proved as Corollary 5.1 in \cite{OSper}.
If we write the power series for $p$ and $q$ near $z_0$ as
\begin{equation} \label{pspq}
    p(z)-p(z_0)=-\sum_{s=0}^\infty p_s (z-z_0)^{s+\mu}, \qquad q(z)=\sum_{s=0}^\infty q_s (z-z_0)^{s}
\end{equation}
then  $\alpha_{n}(p,q;z_0)$ may be given in terms of these coefficients.
This requires
 the {\em partial ordinary Bell polynomials},  which are defined with the generating function
\begin{equation} \label{pobell2}
    \left( p_1 x +p_2 x^2+ p_3 x^3+ \cdots \right)^j = \sum_{i=j}^\infty \hat{B}_{i,j}(p_1, p_2, p_3, \dots) x^i.
\end{equation}
Each $\hat{B}_{i,j}(p_1, p_2, p_3, \dots)$ is a polynomial in $p_1, p_2, \dots, p_{i-j+1}$ of homogeneous degree $j$ with positive integer coefficients.
Then we have
\begin{equation} \label{hjw}
    \alpha_n(p,q;z_0) = \frac{1}{\mu} p_0^{-(n+1)/\mu} \sum_{i=0}^n q_{n-i}\sum_{j=0}^i  \binom{-(n+1)/\mu}{j}  \hat B_{i,j}\left(\frac{p_1}{p_0},\frac{p_2}{p_0},\cdots\right).
\end{equation}
This formulation is essentially due to
 Campbell,  Fr{\"o}man and Walles, and was rediscovered also by Wojdylo. It is straightforward to derive from Perron's formulas; see \cite[Prop. 7.2]{OSper} and the references there.

\subsection{Proof of Theorem \ref{mainthm}} \label{prft}
Recall the definitions of $p(z)$ and $f_\ell(z)$ in \e{pgdef}. Put
\begin{equation} \label{gzdef}
  g(z)  := \frac 1{2\pi i}\left( \frac{- z}{2\pi (1-e^{ z})}\right)^{1/2}
\end{equation}
and let $\mathcal L = L_4$. Then by Theorem \ref{coco} and \e{jan}
\begin{multline} \label{form}
   A_m(1,N) = \frac{2 \Re}{N^{3/2}} \int_{\mathcal L} \frac { e^{z/N}}{ (e^{z/N}-1)^{m+1}}  \cdot g(z) \cdot \exp\left(N \cdot p(z)+  \sum_{\ell =1}^{L-1} \frac{f_\ell(z)}{N^{\ell}}+h_L(z,N)\right) \, dz \\
  +O( N^{m} e^{0.05 N})
\end{multline}
with $h_L(z,N) \ll N^{-L}$ and the implied constant in \e{form} depending only on $m \in \Z$.
\begin{proof}[Proof of Theorem \ref{mainthm}]
We will apply Theorem \ref{sdle} to \e{form} after getting it into the right form. With our
 $p(z)$  and saddle-point $z_0$ in \e{zeros}, we find $\mu=2$, $p_0\approx -0.013-0.0061i$ and the steepest-descent angles are $\theta_0\approx -1.347$ and $\theta_1=\pi+\theta_0$ in the notation of Assumptions \ref{asma}. By approximating $p(z)$ and its first two derivatives on $\mathcal L$ it may be shown that condition \e{<} holds. This is \cite[Corollary 5.5]{OS16}. It follows that
\begin{equation} \label{w0bnd}
  \left| \exp\left(N \cdot p(z)\right) \right| = \exp\left(N \cdot \Re(p(z))\right) \lqs \exp\left(N \cdot \Re(p(z_0))\right) = |w_0|^{-N}
\end{equation}
for all $z\in \mathcal L$, using the identity after \e{zeros}.

The bound $h_L(z,N) \ll N^{-L}$ implies
$
  \exp\left( h_L(z,N)\right) = 1+O(N^{-L})$.
Also for $z\in \mathcal L$
\begin{equation*}
  \frac{1}{N^{3/2}}  \frac { e^{z/N}}{ (e^{z/N}-1)^{m+1}}  \cdot g(z) \cdot \exp\left(  \sum_{\ell =1}^{L-1} \frac{f_\ell(z)}{N^{\ell}}\right) \ll N^{m-1/2}
\end{equation*}
by Corollary \ref{cocor} and Lemma \ref{flem},
and so $h_L(z,N)$ may be removed  from \e{form} at the expense of a total error of size
$
  O\left( |w_0|^{-N} N^{m-1/2-L} \right).
$
 Hence
\begin{equation} \label{form2}
 A_m(1,N) = \frac{2 \Re}{N^{3/2}} \int_{\mathcal L} e^{N \cdot p(z)}g(z) \frac { e^{z/N}}{ (e^{z/N}-1)^{m+1}}    \exp\left( \sum_{\ell =1}^{L-1} \frac{f_\ell(z)}{N^{\ell}}\right) \, dz
  +O\left( \frac{|w_0|^{-N}}{ N^{L-m+1/2}} \right)
\end{equation}
since  $|w_0|^{-N} \approx e^{0.068N}$ by \e{uv} implies (for an implied constant depending on $L$ and $m$) that
\begin{equation*}
  N^{m} e^{0.05 N} \ll \frac{|w_0|^{-N}}{ N^{L-m+1/2}}.
\end{equation*}

The next lemma expands the right two factors of the integrand into powers of $N$.

\begin{lemma} \label{eng}
Suppose $z\in \mathcal L$, $m\in \Z$ and $N, L \in \Z_{\gqs 1}$. Then for functions $\g_{m,j}(z)$ given in \e{vstar},
\begin{equation*}
  \left(e^{z/N} -1\right)^{-m -1} \exp\left(\frac zN+\sum_{\ell =1}^{L-1} \frac{f_\ell(z)}{N^{\ell}}\right)
 = N^{m+1}\Bigg\{\sum_{j=0}^{d-1} \frac{\g_{m,j}(z)}{N^j} + O\left(\frac 1{N^{d}}\right)\Bigg\}
\end{equation*}
when $0 \lqs d \lqs L$. The implied constant depends only on $m$ and  $L$.
\end{lemma}
\begin{proof} Recall that $\mathcal L = L_4$ is the line segment between $2\pi c$ and $3\pi c$. Hence if $z\in \mathcal L$ then $|z|<10$ and $\Re(z)<-1$.
It follows from \e{nor} and Proposition \ref{taylor} that
\begin{equation}\label{ty}
  \left( \frac{z/N}{e^{z/N}-1}\right)^{m+1} =\sum_{r=0}^{M-1} \frac{B^{(m+1)}_r}{r!} \frac {z^r}{N^r} + O\left(\frac 1{N^{M}}\right)
\end{equation}
for $N\gqs 4$ since that implies $|z|/N \lqs \pi$. The implied constant in \e{ty} depends only on $m$ and $M$. By increasing it we may ensure that \e{ty} is true for all $N\gqs 1$.

Now consider
\begin{equation} \label{fw}
  \exp\left( z\cdot w+\sum_{\ell =1}^{L-1} f_\ell(z)\cdot w^{\ell}\right) = \sum_{j=0}^\infty u_j(z) \cdot w^j,
\end{equation}
where we have replaced $1/N$ by $w$ and obtained an entire function of $w$. If $0\lqs j \lqs L-1$ then
\begin{equation} \label{uiz}
    u_{j}(z)=\sum_{m_1+2m_2+3m_3+ \dots +j m_j =j}\frac{(z+f_1(z))^{m_1}}{m_1!}\frac{f_2(z)^{m_2}}{m_2!} \cdots \frac{f_j(z)^{m_j}}{m_j!},
\end{equation}
with $u_{0}(z)=1$. Applying Proposition \ref{taylor} to \e{fw} with $R=2$, say, shows
\begin{equation} \label{ty2}
  \exp\left(\frac zN+\sum_{\ell =1}^{L-1} \frac{f_\ell(z)}{N^{\ell}}\right) = \sum_{j=0}^{d-1} \frac{u_j(z)}{N^j} + O\left(\frac 1{N^{d}}\right)
\end{equation}
for $d\lqs L$. Define
\begin{equation} \label{vstar}
    \g_{m,j}(z):=\sum_{r=0}^j B_r^{(m+1)} \frac{z^{r-m-1}}{r!}\cdot u_{j-r}(z)
\end{equation}
and the lemma follows from the product of \e{ty} and \e{ty2}.
\end{proof}

Using Lemma \ref{eng} with $d=L$ in \e{form2} shows, for an implied constant depending only on $m$ and $d$,
\begin{equation*}
   \frac{A_m(1,N)}{N^{m+1}} = \sum_{j=0}^{d-1}  \frac{2 \Re}{N^{j+3/2}} \int_{\mathcal L} e^{N \cdot p(z)} \cdot g(z) \cdot \g_{m,j}(z) \, dz
  +O\left( \frac{|w_0|^{-N}}{ N^{d+3/2}} \right).
\end{equation*}
Applying Theorem \ref{sdle} to these integrals gives
\begin{equation*}
  \int_{\mathcal L} e^{N \cdot p(z)} \cdot g(z) \cdot \g_{m,j}(z) \, dz =
-e^{N \cdot p(z_0)}\left(\sum_{s=0}^{S-1}\G\left(s+\frac 12 \right) \frac{2\alpha_{2s}(g \cdot \g_{m,j})}{N^{s+1/2}}+O\left( \frac{1}{N^{S+1/2}}\right) \right)
\end{equation*}
 since $r=1$. We have written $\alpha_{2s}(g \cdot \g_{m,j})$ instead of $\alpha_{2s}(p,g \cdot \g_{m,j};z_0)$  for brevity, since $p$ and $z_0$ are fixed.
Choose $S=d$ to give a small enough error and therefore
\begin{multline} \label{jmp}
      \frac{A_m(1,N)}{N^{m+1}}  = -\sum_{j=0}^{d-1}  \frac{2 \Re}{N^{j+3/2}}  e^{N \cdot p(z_0)} \sum_{s=0}^{d-1}\G\left(s+\frac 12 \right) \frac{2\alpha_{2s}(g \cdot \g_{m,j})}{N^{s+1/2}}
+ O\left( \frac{|w_0|^{-N}}{N^{d+3/2}}\right)\\
     = -\Re \Bigg[  w_0^{-N}
    \sum_{j=0}^{2d-2} \frac{4}{N^{j+2}}    \sum_{s=\max(0,j-d+1)}^{\min(j,d-1)} \G\left(s+\frac 12 \right)  \alpha_{2s}(g \cdot \g_{m,j-s})
    \Bigg]+ O\left( \frac{|w_0|^{-N}}{N^{d+3/2}}\right) \\
     = -\Re \Bigg[  w_0^{-N}
    \sum_{j=0}^{d-2} \frac{4}{N^{j+2}}    \sum_{s=0}^{j} \G\left(s+\frac 12 \right)  \alpha_{2s}(g \cdot \g_{m,j-s})
    \Bigg]+ O\left( \frac{|w_0|^{-N}}{N^{d+1}}\right)
\end{multline}
for implied constants depending only on $m \in \Z$ and $d \in \Z_{\gqs 2}$.
This proves \e{maineq} with
\begin{equation}\label{c01ell}
  c_{m,j} = -4\sum_{s=0}^{j} \G\left(s+\frac 12 \right)  \alpha_{2s}(p,g \cdot \g_{m,j-s};z_0).
\end{equation}

Unwinding our definitions lets us compute $c_{m,0}$ as follows. First note that for the Bell polynomials,
$\hat{B}_{j,j}(p_1, p_2, \dots) =p_1^j$. Hence $\alpha_0(p,q;z_0)=p_0^{-1/2}q(z_0)/2$ by \e{hjw} and so
\begin{align}
  c_{m,0} & = -4 \G(1/2) \alpha_0(p,g \cdot \g_{m,0};z_0) \notag\\
   & = -4\sqrt{\pi} p_0^{-1/2}g(z_0)\g_{m,0}(z_0)/2 \notag\\
& = -2\sqrt{\pi} p_0^{-1/2} g(z_0) \cdot  z_0^{-m-1}. \label{oou}
\end{align}
Note that $\frac d{dz} \li(z)=z^{-1}\pl_1(z)= -z^{-1} \log(1-z)$, implying
\begin{equation*}
  p'(z)=-\frac 1z\left( p(z)+\log(1-e^z)\right), \qquad p''(z)= -\frac 2z p'(z)+\frac{e^z}{z(1-e^z)}.
\end{equation*}
Hence, with $w_0 = 1-e^{z_0}$,
\begin{equation*}
  p_0= -\frac 12 p''(z_0)=-\frac{e^{z_0}}{2 z_0 w_0}
\end{equation*}
and finally, using definition \e{gzdef},
\begin{equation}  \label{p0q0}
    p_0^{1/2}  = -\frac{ i e^{ z_0/2}}{\sqrt{2} z_0^{1/2} w_0^{1/2}},\qquad
    g(z_0)   = -\frac{ z_0^{1/2}}{(2\pi)^{3/2} w_0^{1/2}}.
\end{equation}
Substituting \e{p0q0} into \e{oou} gives $c_{m,0}$ and  completes the proof of Theorem \ref{mainthm}.
\end{proof}

With further work, $c_{m,j}$ may be given explicitly for higher $j$ values. For example
\begin{equation}\label{cm1}
  c_{m,1} =\frac{1}{2\pi i \cdot z_0^{m-1}}\left\{ \frac{m-1}{e^{z_0/2}}+\frac{w_0}{e^{3z_0/2}}\left(
\frac 16 +\frac{m-1}{z_0}+\frac{m(m-1)}{z_0^2}
\right)\right\}
\end{equation}
and this is equivalent to \cite[(5.41)]{OS16}.

We briefly mention here the method used in \cite{OS16,OS16b} which is much more indirect, but leads to the same asymptotics. For example, with $\xi=1$ and replacing $\tau/N$ by $\tau$ in \e{fary}, we have
\begin{equation*}
   A_m(1,N) =  \int_{\mathcal C} \frac { e^{2\pi i \tau}}{(e^{2\pi i \tau}-1)^{m+1} ( e^{2\pi i \tau})_N} \, d\tau
\end{equation*}
with $\mathcal C$ encircling only the pole of order $N$ at $0$. Since the integrand has period $1$, it turns out that the sum of the residues  at all the poles of the integrand in $[0,1)$ is zero. These poles occur at the Farey fractions of order $N$. A subset of simple poles is identified that make a large contribution to this zero sum and the asymptotics of these simple residues match those found in Theorem \ref{mainthm}, but with the opposite sign. It is then shown in \cite[Thm. 1.4]{OS16} that $A_m(1,N)$ plus all the coefficients corresponding to poles of order greater than $N/100$ have the asymptotics of Theorem \ref{mainthm}.

\section{The asymptotics of $A_m(\xi,N)$ as $N \to \infty$} \label{xiasy}
The  previous results are extended to expansions around any primitive $k$th root of unity $\xi$ in this section. The coefficients $A_m(\xi,N)$ are expressed by \e{qsz2} in terms of the integrals
\begin{equation} \label{grand}
\int_{\mathcal D'} \frac { e^{z/N}}{ (e^{z/N}-1)^{m+1} (\rho \cdot  e^{z/N})_N} \, dz 
\end{equation}
for $\mathcal D'$ the top half of a circle of radius less than $2\pi/k$ and with $\rho$ equaling $\xi$ and $\overline{\xi}$.

\subsection{Estimating $1/(\rho \cdot e^{z/N})_N$}
Recall the Bernoulli polynomials $B_n(\lambda)$ generated by
\begin{equation} \label{bep}
 \frac{z e^{\lambda z}}{e^z-1} =\sum_{n=0}^\infty B_n(\lambda) \frac{z^n}{n!} \qquad (|z|<2\pi).
\end{equation}
The function $p(z)$ is as before in \e{pgdef}. We make the new definitions 
\begin{equation}\label{qro}
  g_\rho(z)  := \left( \frac{-z}{2\pi(1-e^{z})}\right)^{1/2} \times \prod_{j=1}^{k-1} \left( \frac{1-\rho^{-j}e^{z}}{1-\rho^{-j}}\right)^{j/k-1/2}
\end{equation}
and for $\ell\in \Z_{\gqs 1}$,
\begin{multline}\label{frm}
  f_{\rho,\ell}(z) := \frac{(-1)^{\ell+1} ( k z)^{\ell}}{(\ell+1)!}\Bigg(\delta_{1,\ell} \frac 1{12}+
B_{\ell+1} \pl_{1-\ell}(e^{z})  \\
    +\sum_{j=1}^{k-1}B_{\ell+1}(j/k)\Big\{
\pl_{1-\ell}(\rho^{-j} e^{z}) - \pl_{1-\ell}(\rho^{-j} )\Big\}\Bigg).
\end{multline}
With this notation, the following result generalizes Theorem \ref{coco}.

\begin{theorem} \label{coco2}
Let $\rho$ be a primitive $k$th root of unity. Suppose $N\gqs 1$ and $k \mid N$. Then for all $z\in \C$ with $\Re(z) \lqs -\delta <0$ and  $|z|\lqs C$ we have
\begin{equation*}
  \frac{1}{(\rho \cdot e^{z/N})_N}  =
\left( \frac{k}{N}\right)^{1/2} g_\rho(z)
\exp\left(\frac Nk \cdot p(k z) +  \sum_{\ell =1}^{L-1} \frac{f_{\rho,\ell}(z)}{N^{\ell}}+O\left(N^{-L}\right)\right)
\end{equation*}
for an implied constant depending only on $L \in \Z_{\gqs 1}$, $k$, $\delta$ and $C$.
\end{theorem}
\begin{proof}
As in the beginning of the proof of Theorem \ref{coco}, for $\Re(z)<0$,
\begin{align*}
  \log\left(1/(\rho \cdot e^{z/N})_N \right)  & =   -  \sum_{j=1}^N \log(1- \rho^j e^{jz/N}) \\
   & = \sum_{j=1}^N \sum_{r=1}^\infty \frac{1}{r} \rho^{rj} e^{r jz/N}
 =  \sum_{a=0}^{k-1} \sum_{r=1}^\infty \frac{\rho^{-r a}}r \sum_{\substack{1\lqs j \lqs N  \\ j \equiv -a \bmod k}} e^{r jz/N}.
\end{align*}
We are assuming $k \mid N$ so that
\begin{equation*}
  \sum_{\substack{1\lqs j \lqs N  \\ j \equiv -a \bmod k}}  X^j
= X^{-a} \frac{1-X^N}{X^{-k}-1}.
\end{equation*}
Then
\begin{equation*}
   \log\left(1/(\rho \cdot e^{z/N})_N \right)
= \sum_{a=0}^{k-1} \sum_{r=1}^\infty \frac{\rho^{-r a}}r e^{-ra z/N} \frac{1-e^{r z}}{e^{-k r z/N}-1}.
\end{equation*}
For fixed $k$ and $\rho$, with $\rho$ a primitive $k$th root of unity, put
\begin{equation*}
  \Phi_a(z,N)  := \sum_{r=1}^\infty \frac{\rho^{-r a}}r  \frac{e^{ra z/N}}{e^{k r z/N}-1}, \qquad
  \chi_a(z,N)  := \sum_{r=1}^\infty \frac{\rho^{-r a}}r  \frac{e^{-rz} \cdot e^{ra z/N}}{e^{k r z/N}-1},
\end{equation*}
for $0\lqs a \lqs k-1$. Then
\begin{equation} \label{radu}
  1/(\rho \cdot e^{z/N})_N = \exp\bigg( \sum_{a=0}^{k-1} \Big\{ \Phi_a(-z,N) - \chi_a(-z,N)\Big\}\bigg).
\end{equation}

\begin{prop} \label{and}
Suppose $N$, $L \in \Z_{\gqs 1}$ and $\Re(z)<0$. Then
\begin{equation} \label{ph0}
  \Phi_0(-z,N) = -\frac{\pi^2 N}{6k z}+ \frac {k z}{24N}  +\frac 12 \log\left(\frac{- k z}{2\pi N} \right) +O\left( \frac {|z|^L}{N^L} \right)
\end{equation}
and for $1\lqs a \lqs k-1$
\begin{equation} \label{pha}
  \Phi_a(-z,N) = -\sum_{\ell=0}^{L} \frac{(-1)^\ell}{\ell!} \left(\frac{ k z}{N}\right)^{\ell-1} B_{\ell}(a/k) \cdot \pl_{2-\ell}(\rho^{-a})
 +O\left( \frac {|z|^L}{N^L} \right)
\end{equation}
with implied constants depending only on $k$ and $L$.
\end{prop}
\begin{proof}
The estimate \e{ph0} follows easily from Proposition \ref{phi} since $\Phi_0(z,N) = \Phi(k z,N)$. We may now assume $1\lqs a \lqs k-1$. Recall the Hurwitz zeta function
\begin{equation*}
  \zeta(s,q):=\sum_{n=0}^\infty \frac 1{(n+q)^s} \qquad (\Re(s)>1,\ 0<q\lqs 1).
\end{equation*}
See for example \cite[Chap. 12]{ApIntro} for all the properties of $\zeta(s,q)$ used in this proof, including the following identity,
\begin{equation*}
  \int_0^\infty \frac{u^{s-1}e^{\theta u}}{e^u-1}\, du = \G(s)\zeta(s,1-\theta) \qquad (\Re(s)>1, \ 0\lqs \theta<1).
\end{equation*}
Similarly to \e{mell}, we then obtain the Mellin transform of $\Phi_a(z,N)$ with respect to $N$ as
\begin{equation} \label{mell2}
  \mathcal{M}\Phi_a(z,\cdot)(s) = ( k z)^s \G(-s) \cdot \zeta(-s,1-a/k) \cdot \pl_{1-s}(\rho^{-a})
\end{equation}
for $\Re(s)<-1$. The Hurwitz zeta has a  continuation to all $s\in \C$ and is analytic except for  a simple pole at $s=1$ with residue $1$. By Jonqui\`ere's formula, initially for $\Re(s)<0$,
\begin{equation} \label{hf}
  \pl_{1-s}(e^{2\pi i \theta}) = \frac{\G(s)}{(2\pi)^s}\left(e^{\pi i s/2} \zeta(s,\theta)+e^{-\pi i s/2} \zeta(s,1-\theta) \right)
\end{equation}
for $0<\theta<1$. Then \e{hf} gives the continuation of its left side to $s\in \C$. The only possible poles are at $s=0$ (from $\G(s)$) and $s=1$ (from the Hurwitz zetas). The values of $\zeta(s,q)$ for $s\in \Z_{\lqs 0}$ are
\begin{equation}\label{zvs}
  \zeta(-n,q) = -\frac{B_{n+1}(q)}{n+1}
\end{equation}
for $B_m(q)$ the $m$th Bernoulli polynomial. As $B_1(q)=q-1/2=-B_1(1-q)$ we find that $\pl_{1-s}(e^{2\pi i \theta})$ does not have a pole at $s=0$. The residues at $s=1$ also cancel, implying there is not a pole at $s=1$ either and so $\pl_{1-s}(e^{2\pi i \theta})$ is an entire function of $s$ when $0<\theta<1$.

The bounds in \cite[Thm. 12.23]{ApIntro} imply that for any $K\gqs 0$
\begin{equation*}
  \zeta(s,q) \ll |\Im(s)|^{K+2} \qquad (\Re(s)\gqs -K, \ |\Im(s)|\gqs 1).
\end{equation*}
It follows from \e{hf} and Stirling's formula that $\pl_{1-s}(e^{2\pi i \theta})$ also has at most polynomial growth in $|\Im(s)|$ as $|\Im(s)|\to \infty$. Therefore, applying Stirling's formula once more to $\G(-s)$, we see that \e{mell2} has exponential decay in $|\Im(s)|$ as $|\Im(s)|\to \infty$.
Mellin inversion then yields
\begin{equation*}
  \Phi_a(z,N) = \frac{1}{2\pi i} \int_{(c)}  \G(-s) \cdot \zeta(-s,1-a/k) \cdot \pl_{1-s}(\rho^{-a}) \left( \frac{k z}N \right)^s \, ds,
\end{equation*}
for $c<-1$ and moving the line of integration right, past $\Re(s)=L$, picks up residues of the simple poles at $s=-1,0, \dots, L$.  Then \e{pha} follows where we used $B_m(q) = (-1)^m B_m(1-q)$.
\end{proof}

Note that the analytically continued $\pl_{1-s}(z)$ in \e{hf} with $|z|=1$, $z\neq 1$ agrees with the expressions \e{cot}, \e{fro} when $s$ is a positive integer.

The next result has a similar proof to Proposition \ref{chi}, using \e{bep}.


\begin{prop} \label{and2}
Suppose $N\gqs 1$, $\Re(z) \lqs -\delta <0$ and  $|z|\lqs C$. Then
\begin{equation*}
  -\chi_a(-z,N) = \sum_{\ell=0}^{L} \frac{(-1)^\ell}{\ell!} \left(\frac{ k z}{N}\right)^{\ell-1} B_{\ell}(a/k) \cdot \pl_{2-\ell}(\rho^{-a} e^{z})
+O\left( \frac {1}{N^L} \right)
\end{equation*}
for an implied constant depending only on $L \in \Z_{\gqs 1}$, $k$, $\delta$ and $C$.
\end{prop}

For the sum on the right of \e{radu}, Propositions \ref{and} and \ref{and2} imply
\begin{multline} \label{ge}
  \Phi_0(-z,N)- \chi_0(-z,N) = \frac{N}{kz}\left(\li(e^z)-\li(1) \right) +\frac 12 \log\left(\frac{- k z}{2\pi N} \right) -\frac 12 \log\left(1-e^z \right) \\
  + \frac {k z}{24N}
+\sum_{\ell=2}^{L} \frac{1}{\ell!} \left(\frac{ k z}{N}\right)^{\ell-1} B_{\ell} \cdot \pl_{2-\ell}( e^{z})  +O\left( \frac {1}{N^L} \right)
\end{multline}
and
\begin{multline} \label{ge2}
  \sum_{a=1}^{k-1} \Big\{ \Phi_a(-z,N) - \chi_a(-z,N)\Big\} \\
= \sum_{\ell=0}^{L} \frac{(-1)^\ell}{\ell!} \left(\frac{ k z}{N}\right)^{\ell-1}
 \sum_{a=1}^{k-1}  B_{\ell}(a/k) \Big\{\pl_{2-\ell}(\rho^{-a} e^{z}) - \pl_{2-\ell}(\rho^{-a}) \Big\}
  +O\left( \frac {1}{N^L} \right).
\end{multline}
The coefficient of $N/k$ in \e{ge}, \e{ge2} is
\begin{equation} \label{hnj}
  \frac 1z \sum_{a=0}^{k-1}   \Big\{\pl_{2}(\rho^{-a} e^{z}) - \pl_{2}(\rho^{-a}) \Big\}
\end{equation}
and may be simplified  as follows.
If $\omega^k=1$ for a positive integer $k$ then
\begin{equation*}
  1+\omega+\omega^2+ \cdots + \omega^{k-1} =   \begin{cases}
                                                 k, & \hbox{if $\omega =1$;} \\
                                                 0, & \hbox{if $\omega \neq 1$.}
                                               \end{cases}
\end{equation*}
For $\rho$ a primitive $k$th root of unity we infer that
\begin{equation} \label{licl}
  \sum_{j=0}^{k-1} \li(\rho^j \cdot z)  = \frac 1k \li(z^k), \qquad \quad
  \sum_{j=0}^{k-1} \cl\left(\theta+ \frac{2\pi j}k\right)  = \frac 1k \cl(k \theta).
\end{equation}
The first identity in  \e{licl} is the well-known {\em distribution property} of the dilogarithm \cite[p. 9]{Zag07} and shows that \e{hnj} equals $p(kz)$. The second identity in \e{licl} for the Clausen function will be required later; it is a consequence of the first because $\cl(\theta) = \Im(\li (e^{i \theta}))$.

Next, recall that $\pl_1(z)=-\log(1-z)$ and $B_1(z)=z-1/2$. Hence the log terms of \e{ge}, \e{ge2} combine to give $\sqrt{k/N} g_\rho(z)$ after exponentiation. Finally, gathering the coefficients of negative powers of $N$ yields \e{frm}. This completes the proof of Theorem \ref{coco2}
\end{proof}

\subsection{Moving the path of integration and bounding the error}
Let $r_k(z):=p(kz)/k$. Theorem \ref{coco2} shows that the largest factor of the integrand \e{grand} is $\exp(N \cdot r_k(z))$. Then $r_k'(z)=0$ for the saddle-point $z=z_0/k$ and we move the path of integration to $\mathcal D''$ scaled by a factor $1/k$. The path $\mathcal D''$ was used in sections \ref{moi}, \ref{bne} and is made up of the seven line segments $L_j$ displayed in Figure \ref{bfig}.

It is straightforward to show the next result, with a similar proof to Theorem \ref{ban}.
\begin{theorem} \label{ban2}
 Assume that  $N \gqs 1$ and $k \mid N$ with $\rho$ a primitive $k$th root of unity. Let $j \in \{1,3,5,6,7\}$. Then for an implied constant depending only on $m \in \Z$ and $k$ we have
\begin{equation} \label{qbxx}
     \int_{L_j/k}\frac { e^{z/N}}{ (e^{z/N}-1)^{m+1} (\rho \cdot  e^{z/N})_N} \, dz =  O\left( N^{m+1} e^{0.05 N/k}\right).
\end{equation}
\end{theorem}

Proving \e{qbxx} for the segment $L_2/k$, crossing the imaginary axis, needs a slightly more elaborate version of Proposition \ref{jpp}.
The form of Euler-Maclaurin summation required is given in the next lemma.

\begin{lemma} \label{t1N}
For $t \in [1,N]$, suppose $\psi(t)$ is continuously differentiable  with $|\psi(t)|\lqs C$. Let $N\in \Z_{\gqs 1}$ be divisible by $k$. Then for any integer $m$
\begin{equation} \label{ke}
  \sum_{\substack{1\, \leqslant \, j \, \lqs \, N  \\ j \, \equiv \, m  \bmod k}} \psi(j) = \frac 1k  \int_{1}^{N} \psi(t)\, dt +E
\end{equation}
where
\begin{equation*}
  |E| < 2C+ \frac 12 \int_1^N |\psi'(t)| \, dt.
\end{equation*}
\end{lemma}
\begin{proof}
Choose $v$ so that $0\lqs v \lqs k-1$ with $-v \equiv m \bmod k$ and put $\phi(r):=\psi(rk-v)$. Then the left side of \e{ke} is
\begin{align*}
   \sum_{r=1}^{N/k} \phi(r)
   & = \int_1^{N/k} \phi(t)\, dt + \frac 12(\phi(N/k)+\phi(1)) + \int_1^{N/k} \left(t-\lfloor t \rfloor-\frac 12\right) \phi'(t)\, dt\\
 & = \frac 1k \int_{k-v}^{N-v} \psi(t)\, dt + \frac 12(\psi(N-v)+\psi(k-v)) + E_1
\end{align*}
where
\begin{equation*}
  |E_1| \lqs \frac 12 \int_{k-v}^{N-v} |\psi'(t)| \, dt.
\end{equation*}
Also
\begin{equation*}
   \frac 1k \int_{1}^{N} \psi(t)\, dt  = \frac 1k \int_{k-v}^{N-v} \psi(t)\, dt +\frac 1k\left(\int_{1}^{k-v} \psi(t)\, dt + \int_{N-v}^{N} \psi(t)\, dt\right)
\end{equation*}
and
\begin{equation*}
  \frac 1k\left|\int_{1}^{k-v} \psi(t)\, dt + \int_{N-v}^{N} \psi(t)\, dt\right|  \lqs \frac{k-v-1}k C + \frac{v}k C <C.
\end{equation*}
The lemma follows.
\end{proof}


\begin{prop} \label{jpp2}
Let $\rho$ be a primitive $k$th root of unity. For $1\lqs N$, $k\mid N$ and $0<\theta <2\pi /k$   we have
\begin{equation} \label{jpb}
  \prod_{j=1}^N \left| \rho^j e^{ i j \theta /N}-1\right|^{-1} \ll N^{4k} \exp\left( \frac Nk \cdot \frac {\cl( k \theta)}{ k \theta}\right)
\end{equation}
for an implied constant depending only on $k$ and $\theta$.
\end{prop}
\begin{proof}
Write $\rho=e^{2\pi i h/k}$. As in Proposition \ref{jpp}
\begin{equation} \label{wwcu}
  \log \prod_{j=1}^N \left| e^{2\pi i j h/k +i j \theta/N}-1\right| = \sum_{j=1}^N \log  \left| 2 \sin \left(\pi \frac {j h}k +\frac{j \theta}{2N}\right) \right|.
\end{equation}
The sum over $j$ is broken into residue classes modulo $k$. Write $v' \equiv -v h \bmod k$ for $0\lqs v, v'\lqs k-1$. Then \e{wwcu} is
\begin{equation*}
  \sum_{v=0}^{k-1} \sum_{\substack{1\leqslant j \lqs N  \\ j \equiv -v \bmod k}} \log  \left( 2 \sin \left(\pi \frac{v'}k + \frac{j \theta}{2N}\right)\right)
\end{equation*}
and the absolute value was dropped since $0<\theta <2\pi /k$ implies the argument of sine is in the interval $(0,\pi)$:
\begin{equation} \label{opp}
 0< \frac{\theta}{2N} \lqs \pi \frac{v'}k + \frac{j \theta}{2N} \lqs \pi\frac {k-1}k + \frac{ \theta}{2} <\pi.
\end{equation}
Let
\begin{equation*}
  \psi(t):=\log  \left( 2 \sin \left(\pi \frac{v'}k + \frac{t \theta}{2N}\right)\right).
\end{equation*}
Then by \e{opp}, for $1\lqs t \lqs N$, we have $|\psi(t)| \lqs C$ with
\begin{equation*}
  C:=\max\left\{ \bigg| \log\Big( 2 \sin \Big( \frac{\theta}{2N} \Big)\Big)\bigg|, \
\log 2, \  \bigg| \log\Big( 2 \sin \Big( \pi\frac {k-1}k + \frac{ \theta}{2}\Big)\Big)\bigg|\right\} =\log N +O(1).
\end{equation*}

By Lemma \ref{t1N},
\begin{equation*}
  \sum_{\substack{1\leqslant j \lqs N  \\ j \equiv -v \bmod k}}  \psi(j) = \frac 1k  \int_{1}^{N} \psi(t)\, dt +E
\end{equation*}
where
\begin{equation*}
  |E| < 2C+ \frac 12 \int_1^N |\psi'(t)| \, dt <4C.
\end{equation*}
The last inequality follows as in \e{ex}, \e{ex2}, since $\psi'(t)$ is decreasing.
Next
\begin{equation*}
  \int_{1}^{N} \psi(t)\, dt = -\frac{N}{\theta}\bigg[ \cl\Big( N \cdot \frac \theta N+2\pi\frac{v'}k  \Big) - \cl\Big(1 \cdot \frac \theta N+2\pi\frac{v'}k  \Big)\bigg]
\end{equation*}
and so
\begin{multline*}
  -\log \prod_{j=1}^N \left| e^{2\pi i j h/k +i j \theta/N}-1\right| \lqs  \sum_{v=0}^{k-1}\frac{N}{ k \theta}\bigg[
\cl\Big( \theta +2\pi\frac{v'}k  \Big) - \cl\Big( \frac \theta N+2\pi\frac{v'}k  \Big)
\bigg] \\
  + 4k\log N +O(1).
\end{multline*}
The proof is completed by using \e{licl} to show this is bounded by
\begin{equation*}
  \frac{N}{ k^2 \theta}\bigg[ \cl\Big(k \theta \Big) - \cl\Big(  \frac{k \theta}N\Big) \bigg]
  + 4k\log N +O(1). \qedhere
\end{equation*}
\end{proof}

A similar proof to Corollary \ref{polo} now goes through, showing that \e{qbxx} is true for $j=2$. Gathering our results, we have demonstrated that
\begin{equation} \label{big}
  A_m(\xi,N) = \frac 1N  \int_{L_4/k} Q_m(z;\xi,N) \, dz + \frac 1N \overline{\int_{L_4/k} Q_m(z;\overline{\xi},N) \, dz} + O\left( N^{m} e^{0.05 N/k}\right).
\end{equation}

\subsection{Proof of the main theorem} \label{boh}

\begin{proof}[Proof of Theorem \ref{mainthm3x}] 
Let $\mathcal L = L_4$ as before, and our goal is to find the asymptotics of
\begin{equation}\label{afte}
 D_m(\rho,N) := \frac 1N  \int_{\mathcal L/k} Q_m(z;\rho,N) \, dz
\end{equation}
for $\rho=\xi$ or $\overline{\xi}$ as in \e{big}. By Theorem \ref{coco2}, and using reasoning similar to that before \e{form2},  we find  \e{afte} equals
\begin{equation} \label{kkn}
 \frac{\rho^{-m} k^{1/2}}{2\pi i N^{3/2}} \int_{\mathcal L/k} e^{N \cdot r_k(z)}g_\rho(z) \frac { e^{z/N}}{ (e^{z/N}-1)^{m+1}}    \exp\left( \sum_{\ell =1}^{L-1} \frac{f_{\rho,\ell}(z)}{N^{\ell}}\right) \, dz
  +O\left( \frac{|w_0|^{-N/k}}{ N^{L-m+1/2}} \right)
\end{equation}
when $k$ divides $N$. Expanding the right two factors of the integrand, as in Lemma \ref{eng}, shows
for an implied constant depending only on $m$, $k$ and $d$,
\begin{equation} \label{cal}
   \frac{D_m(\rho,N)}{N^{m+1}} =  \frac{\rho^{-m}}{2\pi i } \sum_{j=0}^{d-1}  \frac{k^{1/2}}{N^{j+3/2}} \int_{\mathcal L/k} e^{N \cdot r_k(z)} \cdot g_\rho(z) \cdot \g_{\rho,m,j}(z) \, dz
  +O\left( \frac{|w_0|^{-N/k}}{ N^{d+3/2}} \right).
\end{equation}
Applying Theorem \ref{sdle} and simplifying, as in \e{jmp}, produces
\begin{equation*}
   \frac{D_m(\rho,N)}{N^{m+1}} =  -  w_0^{-N/k} \frac{\rho^{-m}}{2\pi i }
    \sum_{j=0}^{d-2} \frac{2k^{1/2}}{N^{j+2}}    \sum_{s=0}^{j} \G\Big(s+\frac 12 \Big)  \alpha_{2s}(r_k,g_\rho \cdot \g_{\rho,m,j-s};z_0/k)
    + O\left( \frac{|w_0|^{-N/k}}{N^{d+1}}\right).
\end{equation*}
By setting
\begin{equation*}
  e_{m,j}(\rho,0) := -k^{1/2} \frac{\rho^{-m}}{\pi i } \sum_{s=0}^j \G\left(s+\frac 12 \right)  \alpha_{2s}(r_k,g_\rho \cdot \g_{\rho,m,j-s};z_0/k)
\end{equation*}
we have therefore proved \e{maineq3x} when $k \mid N$.

Now let $M$ be any positive integer. To find the asymptotics of $ A_m(\xi,M)$, we may use our previous work by writing $M=N-v$ for some $v$ with $0\lqs v\lqs k-1$ and $k \mid N$.
Going back to \e{qsz2}, we need to estimate
\begin{equation}\label{est}
  \frac 1M  \int_{\mathcal D'} Q_m(w;\rho,M) \, dw = \frac 1M  \int_{\mathcal D'} \frac {\xi^{-m} e^{w/M}}{2\pi i (e^{w/M}-1)^{m+1} (\rho e^{w/M})_M} \, dw
\end{equation}
for $\rho=\xi$ or $\overline{\xi}$ where $\mathcal D'$ is the top half of any circle of radius less than $2\pi/k$.
With the change of variables $w/M=z/N$, \e{est} becomes
\begin{equation*}
 \frac 1N  \int_{\mathcal D'} \frac {\xi^{-m} e^{z/N}}{2\pi i (e^{z/N}-1)^{m+1} (\rho e^{z/N})_M} \, dz \\
  = \frac 1N  \int_{\mathcal D'} Q_m(z;\rho,N) \prod_{j=M+1}^N \left(1-\rho^j e^{jz/N} \right) \, dz.
\end{equation*}
The product in the integrand is $1$ if $v=0$ and otherwise equals
\begin{equation} \label{pal}
  \prod_{r=0}^{v-1} \left(1-\rho^{-r} e^{ z} e^{- r z/N} \right) = \sum_{n=0}^\infty \frac{\varphi_{\rho,n}(z,v)}{N^n}
\end{equation}
where by a straightforward calculation 
\begin{equation} \label{dap}
  \varphi_{\rho,n}(z,v) = \sum_{j_0+j_1+\cdots +j_{v-1}=n} \kappa_{\rho,j_0}(z,0) \kappa_{\rho,j_1}(z,1) \cdots \kappa_{\rho,j_{v-1}}(z,v-1)
\end{equation}
for
\begin{equation} \label{dap2}
  \kappa_{\rho,j}(z,r) := \delta_{j,0}- \rho^{-r} e^{z}\frac{(- r z)^j}{j!}.
\end{equation}
Set $\varphi_{\rho,n}(z,0):=\delta_{n,0}$   to obtain a valid  formula for $v=0$ also.
All our previous work now goes through, as long as $\g_{\rho,m,j}(z)$ in \e{cal}  is replaced by
\begin{equation} \label{dap3}
  \g^*_{\rho,m,j}(z,v):= \sum_{n=0}^j \g_{\rho,m,n}(z) \cdot \varphi_{\rho,j-n}(z,v).
\end{equation}
Hence
\begin{equation} \label{cal2}
   \frac{D_m(\rho,M)}{N^{m+1}} =  \frac{\rho^{-m}}{2\pi i } \sum_{j=0}^{d-1}  \frac{k^{1/2}}{N^{j+3/2}} \int_{\mathcal L/k} e^{N \cdot r_k(z)} \cdot g_\rho(z) \cdot \g^*_{\rho,m,j}(z,v) \, dz
  +O\left( \frac{|w_0|^{-N/k}}{ N^{d+3/2}} \right)
\end{equation}
and by Theorem \ref{sdle},
\begin{equation} \label{cal3}
   D_m(\rho,M) =  w_0^{-N/k}\sum_{j=0}^{d-1}\frac{e^*_{m,j}(\rho,M_k)}{N^{j+1-m}}
+
O\left(\frac{|w_0|^{-N/k}}{N^{d+1-m}}\right)
\end{equation}
for
\begin{equation} \label{cid}
  e^*_{m,j}(\rho,M_k):= -k^{1/2} \frac{\rho^{-m}}{\pi i } \sum_{s=0}^j \G\left(s+\frac 12 \right)  \alpha_{2s}(r_k,g_\rho \cdot \g^*_{\rho,m,j-s}(\cdot,v);z_0/k)
\end{equation}
with $M_k \equiv M \equiv -v \bmod k$.
The right side of \e{cal3} is expressed in powers of $M$, using $N=M+v$ and the binomial theorem, to obtain  \e{maineq3x} with
\begin{equation} \label{buzz}
  e_{m,j}(\rho,M_k) = w_0^{-v/k} \sum_{n=0}^j v^{j-n} \binom{m-1-n}{j-n} e^*_{m,n}(\rho,M_k).
\end{equation}

For the coefficient of the main term
\begin{align*}
   e_{m,0}(\rho,N_k) & = w_0^{-v/k}  e^*_{m,0}(\rho,N_k) \\
   & = - w_0^{-v/k} k^{1/2} \frac{\rho^{-m}}{\pi i } \G(1/2) \cdot \alpha_0( r_k, g_\rho \cdot \g^*_{\rho,m,0}(\cdot,v);z_0/k).
\end{align*}
Note that $p\to r_k$ and $z_0 \to z_0/k$ means the coefficients $p_s$ on the left of \e{pspq} become $p_s \cdot k^{s+1}$. Therefore
\begin{equation*}
  \alpha_0( r_k, g_\rho \cdot \g^*_{\rho,m,0}(\cdot,v);z_0/k) = \frac 12 (p_0 \cdot k)^{-1/2} g_\rho(z_0/k)
\cdot \g^*_{\rho,m,0}(z_0/k,v)
\end{equation*}
and also
\begin{equation*}
  \g^*_{\rho,m,0}(z,v)  =  \g_{\rho,m,0}(z,v) \cdot \varphi_{\rho,0}(z,v)
    = z^{-m-1} \prod_{r=0}^{v-1} \left( 1-\rho^{-r} e^{z}\right)
\end{equation*}
where we understand the above product equals $1$ if $v=0$. Altogether, with \e{p0q0}, \e{qro}
and noting that
\begin{equation*}
  \left( \frac{- z_0/k}{1-e^{z_0/k}}\right)^{1/2} = -\frac{i z_0^{1/2} k^{-1/2}}{(1-e^{ z_0/k})^{1/2}}
\end{equation*}
we find
\begin{multline} \label{cas}
  e_{m,0}(\rho,N_k) =    \frac{-z_0}{2\pi i e^{z_0/2}}
  \frac{(w_0 /k)^{1/2}}{(1-e^{ z_0/k})^{1/2} } \times \prod_{j=1}^{k-1} \left( \frac{1-\rho^{-j}e^{ z_0/k}}{1-\rho^{-j}}\right)^{j/k-1/2}\\
\times  \rho^{-m} ( z_0/k)^{-m -1}
\times w_0^{-v/k} \prod_{j=0}^{v-1} \left( 1-\rho^{-j} e^{ z_0/k}\right) .
\end{multline}
Write the last factors of \e{cas} as
$$
F(v):=w_0^{-v/k} \prod_{j=0}^{v-1} \left( 1-\rho^{-j} e^{ z_0/k}\right),
$$
and the final step is to express this in terms of $N_k$. We have $v=k-N_k$ unless $v=0$ (in which case $N_k=0$).
Note that $\prod_{j=0}^{v-1} ( 1-\rho^{-j}z) =1-z^k$. Then $w_0=1-e^{z_0}$ implies
that $F(v)$ extended to $\Z_{\gqs 0}$ is periodic: $F(v+k)=F(v)$. In particular $F(v)=F(k-N_k)$, and a short calculation then provides
\begin{equation*}
  F(v)=F(k-N_k) = w_0^{N_k/k}\prod_{j=1}^{N_k} \left( 1-\rho^{j} e^{ z_0/k}\right)^{-1}
\end{equation*}
and \e{cas2}.
The proof of  Theorem \ref{mainthm3x} is complete.
\end{proof}

\section{Numerical work} \label{nume}

\subsection{Summary of the calculation of $e_{m,j}(\rho,N_k)$}
As we have seen, the expansion of $1/(q)_N$ about a primitive $k$th root of unity $\xi$ is a Laurent series with $m$th coefficient $A_m(\xi,N)$. The asymptotics of each coefficient as $N\to \infty$ are given by Theorem \ref{mainthm3x} in terms of $e_{\ell,j}(\rho,N_k)$ for $\rho=\xi$ or $\overline{\xi}$. The calculation of these numbers  may be summarized as follows and they are readily programmed.

First define the functions  $p(z)$ in \e{pgdef} and $f_{\rho,\ell}(z)$ in  \e{frm} in terms of polylogarithms. Similarly to \e{uiz} and \e{vstar},
for $j \in \Z_{\gqs 0}$ put
\begin{equation} \label{uiz2}
    u_{\rho,j}(z):=\sum_{m_1+2m_2+3m_3+ \dots =j}\frac{(z+f_{\rho,1}(z))^{m_1}}{m_1!}\frac{f_{\rho,2}(z)^{m_2}}{m_2!} \cdots \frac{f_{\rho,j}(z)^{m_j}}{m_j!},
\end{equation}
with $u_{\rho,0}(z)=1$, and
\begin{equation} \label{vstar2}
\g_{\rho,m,j}(z):=\sum_{r=0}^j B_r^{(m+1)} \frac{z^{r-m-1}}{r!}\cdot u_{\rho,j-r}(z).
\end{equation}
Define $\varphi_{\rho,n}(z,v)$ with \e{dap},  \e{dap2} and as in  \e{dap3} put
\begin{equation*} 
  \g^*_{\rho,m,j}(z,v):= \sum_{n=0}^j \g_{\rho,m,n}(z) \cdot \varphi_{\rho,j-n}(z,v).
\end{equation*}
This lets us take into account the variation of $A_m(\xi,N)$ with $N \bmod k$ as we let $N_k \equiv -v \equiv N \bmod k$ for $0\lqs N_k \lqs k-1$ and $0\lqs v \lqs k-1$.
The product $g_\rho(z)$ is defined in \e{qro}. We set $r_k(z):=p(kz)/k$ and it has a simple saddle-point at $z_0/k$. To compute $\alpha_{2s}(r_k, g_\rho \cdot \g^*_{\rho,m,j}(\cdot,v);z_0/k)$ we use the formula \e{hjw} with $\mu=2$ and coefficients
\begin{equation*}
  p_n=-\frac{1}{(\mu+n)!}\left . \frac{d^{\mu+n} }{dz^{\mu+n}}r_k(z) \right|_{z_0/k}, \qquad q_n=\frac{1}{n!} \left . \frac{d^{n} }{dz^{n}}g_\rho(z) \cdot \g^*_{\rho,m,j}(z,v) \right|_{z_0/k}.
\end{equation*}
As in \e{cid}  set
\begin{equation*}
  e^*_{m,j}(\rho,N_k):= -k^{1/2} \frac{\rho^{-m}}{\pi i } \sum_{s=0}^j \G\left(s+\frac 12 \right)  \alpha_{2s}(r_k,g_\rho \cdot \g^*_{\rho,m,j-s}(\cdot,v);z_0/k)
\end{equation*}
and this is equal to $e_{\ell,j}(\rho,N_k)$ when $N_k=0$. Otherwise, for $N_k \neq 0$,
\begin{equation*} 
  e_{m,j}(\rho,N_k) = w_0^{-v/k} \sum_{n=0}^j v^{j-n} \binom{m-1-n}{j-n} e^*_{m,n}(\rho,N_k).
\end{equation*}

\subsection{Computing $A_m(\xi,N)$ exactly}

Recall the
 Norlund polynomials $B_n^{(\alpha)}$, with generating function \e{nor}, and the Bernoulli numbers $B_n = B_n^{(1)}$. For all $n \in \Z_{\gqs 0}$ and $\alpha \in \C$, the explicit formula
\begin{equation} \label{nor2}
  B_n^{(\alpha)} = \sum_{j=0}^n (-1)^j \binom{\alpha+n}{n-j} \binom{\alpha+j-1}{j} \binom{n+j}{j}^{-1} \stirb{n+j}{j}
\end{equation}
is proved in \cite[Eq. 15]{SrTo88}. It follows that $B_n^{(\alpha)}$ is a polynomial  of degree $n$ in $\alpha$.
Here the Stirling number $\stirb{m}{k}$   denotes the number of partitions of  $m$ elements into $k$ non-empty subsets.

We need a further generalization of the Bernoulli numbers. As described in \cite[Sect. 3]{OS15}, for example, the Apostol-Bernoulli numbers have generating function
\begin{equation}\label{apb}
\frac{z}{\rho e^z - 1} = \sum_{m=0}^\infty \beta_m(\rho) \frac{z^m}{m!} \qquad (\rho \in \C)
\end{equation}
and may be expressed in terms of Bernoulli polynomials (recall \e{bep}) with
\begin{equation*}
  \beta_m(\xi)  =  k^{m-1} \sum_{j=0}^{k-1} B_m(j/k) \cdot  \xi^{j}
\end{equation*}
for $\xi$  a  primitive $k$th root of unity. Clearly $\beta_m(1)=B_m$ for  $m \gqs 0$.

\begin{prop} \label{kop} For all $N \in \Z_{\gqs 1}$, $m \in \Z$ and primitive $k$th roots of unity $\xi$, we have
\begin{equation*} 
 A_m(\xi,N)  = \frac{(-1)^N \xi^{-m}}{ N!} \sum_{u+v+j_1+j_2+ \cdots + j_N = N+m}
    B_{v}^{(m+1)} \cdot  \beta_{j_1}(\xi) \cdot  \beta_{j_2}(\xi^2)  \cdots  \beta_{j_N}(\xi^N) \frac{1^{j_1} 2^{j_2} \cdots
N^{j_N}}{u! v! j_1 ! j_2 ! \cdots j_N!}
\end{equation*}
where the sum is over all $u,v,j_1, \dots, j_N \in \Z_{\gqs 0}$.
\end{prop}
\begin{proof}
Expressing the integral \e{qsz} in terms of the residue at zero we find
\begin{multline} \label{pop}
  A_m(\xi,N)  = \xi^{-m}  \left[ \text{coeff. of }z^{-1} \right] \frac{e^z}{(e^z-1)^{m+1}(\xi e^z)_N}\\
 = \frac{(-1)^N \xi^{-m}}{N!} \left[ \text{coeff. of }z^{N+m} \right]e^z \left(\frac{z}{e^z - 1}\right)^{m+1}
\left(\frac{z}{\xi e^z - 1}\right)\left(\frac{2z
}{\xi^2 e^{2z} - 1}\right) \cdots \left(\frac{N z}{\xi^N e^{Nz} - 1}\right).
\end{multline}
Inserting the power series coefficients completes the proof.
\end{proof}

Similar formulas to Proposition \ref{kop} are given in \cite[Prop. 3.1]{OS15}. It also  follows from Proposition \ref{kop} that $ A_m(\xi,N)$ is in the field $\Q(\xi)$.

A more efficient method to compute $A_m(\xi,N)$ may be given next based on the work in \cite{OS15}.
For fixed $k$ define
\begin{equation*}
    S_{n,r}(m,N):= \delta_{0,r}\cdot (m+1)+\sum_{\substack{1 \lqs j \lqs N , \ j \equiv r \bmod k}} j^n.
\end{equation*}

\begin{prop} \label{kop2} Let $N \in \Z_{\gqs 1}$, $m \in \Z$ and let $\xi$ be a primitive $k$th root of unity.
For $s:=\lfloor N/k \rfloor$, $M:=s+m \gqs 0$ and $N_k \equiv N \bmod k$ with $0\lqs N_k \lqs k-1$,
\begin{multline}
    A_m(\xi,N) = \frac {(-1)^{s} \xi^{-m} }{k^{2s+1} \cdot s!} \Bigg[\prod_{w=N_k+1}^{k-1} (1-\xi^w) \Bigg]
\\
    \times \sum_{1j_1+2j_2+ \cdots + M j_{M} = M}
     \frac{1}{j_1! j_2! \cdots j_M!}
     \left(-m-\frac{N(N+1)}2 - \sum_{r=0}^{k-1}  \frac{ \beta_1(\xi^r) \cdot S_{1,r}(m,N)}{1 \cdot 1!}\right)^{j_1} \\
    \times
     \left(- \sum_{r=0}^{k-1}  \frac{ \beta_2(\xi^r) \cdot S_{2,r}(m,N)}{2 \cdot 2!} \right)^{j_{2}} \cdots
     \left(- \sum_{r=0}^{k-1}  \frac{ \beta_M(\xi^r) \cdot S_{M,r}(m,N)}{M \cdot M!} \right)^{j_{M}}. \label{wavek}
\end{multline}
Note that for $M=0$ we replace the entire sum over $j_1, \dots, j_M$ in \e{wavek} by $1$.
\end{prop}

The proof is a straightforward exercise starting with  \e{pop} and employing the identities
\begin{align*}
  \log \left(\frac{z }{e^{z}-1} \right) & = -z-\sum_{n=1}^\infty  \frac{B_n }{n \cdot n!} z^n, \\
    \log \left( \frac{\rho -1}{\rho e^z -1} \right)  & = -z -\sum_{n=1}^\infty \frac{\beta_n(\rho)}{n \cdot n!} z^n \qquad (\rho \neq 1).
\end{align*}
Glaisher obtained similar results to Proposition \ref{kop2} for Sylvester waves as seen in Theorem 4.3 and (4.8) of \cite{OS15}.
At each stage in the computation of \e{wavek}, the result may be simplified to a degree $k-1$ polynomial in $\xi$ with rational coefficients.

Some examples of Proposition \ref{kop2} are the following.
With $k=N$, $s=1$, $m=-1$ we find $C_{1N1}(N)=-e^{2\pi i/N}/N^2$; see also \cite[Eq. (2.10)]{OS16}.
For $\xi=1$ and $k=1$, \e{wavek} reduces (for $M=N+m$) to
\begin{multline}
    A_m(1,N) = \frac {(-1)^{N}  }{N!}  \sum_{1j_1+2j_2+ \cdots + M j_{M} = M}
     \frac{1}{j_1! j_2! \cdots j_M!}
     \left(-m-\frac{N(N+1)}2 -   \frac{ B_1 \cdot S_{1,0}(m,N)}{1 \cdot 1!}\right)^{j_1} \\
    \times
     \left(-  \frac{ B_2 \cdot S_{2,0}(m,N)}{2 \cdot 2!} \right)^{j_{2}} \cdots
     \left(-   \frac{ B_M \cdot S_{M,0}(m,N)}{M \cdot M!} \right)^{j_{M}}. \label{wavek2}
\end{multline}
Tables \ref{jeb-4} and \ref{jeb-k4} give further examples of Theorems \ref{mainthm}, \ref{mainthm3x} with the exact values of the coefficients shown on the last line and computed with Proposition \ref{kop2}.

\begin{table}[ht]
\centering
\begin{tabular}{ccc}
\hline
 $r$   & Theorem \ref{mainthm} & \\
\hline
 $1$   & $1.97\textcolor{gray}{608009680605866} \times 10^{60}$ & \\
 $3$    & $1.977415\textcolor{gray}{84770913413} \times 10^{60}$ & \\
 $5$   & $1.977415482\textcolor{gray}{75273845} \times 10^{60}$ & \\
 $7$    & $1.977415482931\textcolor{gray}{52401} \times 10^{60}$  &  \\
\hline
\phantom{$A_{-4}(1,2500)$} & $ 1.97741548293140288 \times 10^{60}$ & $A_{-4}(1,2500)$\\
\hline
\end{tabular}
\caption{The approximations of Theorem \ref{mainthm} to $A_{-4}(1,2500)$.} \label{jeb-4}
\end{table}

\begin{table}[ht]
\centering
\begin{tabular}{ccc}
\hline
 $r$   & Theorem \ref{mainthm3x} & \\
\hline
 $1$   & $6.6\textcolor{gray}{99691529339419} \times 10^{17} - 2.3\textcolor{gray}{252380189830248} \times 10^{18}i$ & \\
 $3$   & $6.6511\textcolor{gray}{84519968432} \times 10^{17} - 2.31583\textcolor{gray}{37396379049} \times 10^{18}i$ & \\
 $5$   & $6.6511950\textcolor{gray}{28374644} \times 10^{17} - 2.31583667\textcolor{gray}{55589084} \times 10^{18}i$ & \\
 $7$   & $6.6511950104\textcolor{gray}{70307} \times 10^{17} - 2.315836673\textcolor{gray}{2365613} \times 10^{18}i$  &  \\
\hline
\phantom{$A_{-2}$} & $ 6.651195010459496 \times 10^{17} - 2.3158366731930319  \times 10^{18}i$ & $A_{1}(i,2501)$\\
\hline
\end{tabular}
\caption{The approximations of Theorem \ref{mainthm3x} to $A_{1}(i,2501)$.} \label{jeb-k4}
\end{table}

\section{Sylvester waves}
\begin{proof}[Proof of Theorem \ref{ma2}]
We are using the same arguments as in Section \ref{xiasy} and  may be brief.
Similarly to \e{qsz} we obtain from \e{wv}
\begin{equation} \label{wv2}
 W_k(N,\lambda N) = -\frac 1N \sum_\xi \int_{\mathcal D} \frac {\xi^{-\lambda N} e^{-\lambda z}}{2\pi i  \cdot (\xi e^{z/N})_N} \, dz
\end{equation}
where we are summing over all primitive $k$th roots of unity $\xi$ and $\mathcal D$ is a circle of radius less than $2\pi/k$. Denoting the integrand in \e{wv2} as $R_\lambda(z;\xi,N)$ we find, as in \e{qsz2} where $\mathcal D'$ is the top half of $\mathcal D$,
\begin{align*} 
 W_k(N,\lambda N) & = -\frac 1N \sum_\xi \bigg(\int_{\mathcal D'} R_\lambda(z;\xi,N) \, dz + \overline{\int_{\mathcal D'} R_\lambda(z;\overline{\xi},N) \, dz}\bigg)\\
& = -\frac 2N \Re \bigg[ \sum_\xi \int_{\mathcal D'} R_\lambda(z;\xi,N) \, dz \bigg].
\end{align*}
Our task is now to estimate
\begin{equation} \label{wv4}
 E_\lambda(\xi,N):= -\frac 2N \int_{\mathcal D'} R_\lambda(z;\xi,N) \, dz = -\frac{\xi^{-\lambda N}}{\pi i N}  \int_{\mathcal D'} \frac { e^{-\lambda z}}{(\xi e^{z/N})_N}  \, dz.
\end{equation}

By Theorem \ref{coco2}  we find that \e{wv4} equals
\begin{equation} \label{mvl}
 -\frac{\xi^{-\lambda N} k^{1/2}}{\pi i N^{3/2}} \int_{\mathcal L/k} e^{N \cdot r_k(z)} e^{-\lambda z} g_\xi(z) \exp\left( \sum_{\ell =1}^{L-1} \frac{f_{\xi,\ell}(z)}{N^{\ell}}\right) \, dz
  +O\bigg( \frac{|w_0|^{-N/k}}{ N^{L+3/2}} \bigg)
\end{equation}
when $k$ divides $N$.
Define $g_{\lambda,\xi}(z):=  e^{-\lambda z} g_\xi(z)$ and expand the exponential of the sum over $\ell$ using coefficients $\omega_{\xi,j}(z)$ which are defined as in \e{uiz2} except with $z+f_{\xi,1}(z)$ replaced by $f_{\xi,1}(z)$. Then for an implied constant depending only on $\lambda'$, $k$ and $d$,
\begin{equation} \label{mvl2}
  E_\lambda(\xi,N)  =  -\frac{\xi^{-\lambda N}}{\pi i } \sum_{j=0}^{d-1}  \frac{k^{1/2}}{N^{j+3/2}} \int_{\mathcal L/k} e^{N \cdot r_k(z)} \cdot g_{\lambda,\xi}(z) \cdot \omega_{\xi,j}(z) \, dz
  +O\left( \frac{|w_0|^{-N/k}}{ N^{d+3/2}} \right).
\end{equation}
Applying Theorem \ref{sdle} and simplifying, as in \e{jmp}, produces
\begin{equation*}
   E_\lambda(\xi,N) =    w_0^{-N/k} \frac{\xi^{-\lambda N}}{\pi i }
    \sum_{j=0}^{d-2} \frac{2k^{1/2}}{N^{j+2}}    \sum_{s=0}^{j} \G\left(s+\frac 12 \right)  \alpha_{2s}(r_k,g_{\lambda,\xi} \cdot \omega_{\xi,j-s};z_0/k)
    + O\left( \frac{|w_0|^{-N/k}}{N^{d+1}}\right).
\end{equation*}
Therefore \e{pres} is true when $k \mid N$ (and $N_k=0$) for
\begin{equation} \label{lio}
  a_{\lambda,j}(0,n_k) := \frac{2k^{1/2}}{\pi i }\sum_{s=0}^{j} \G\left(s+\frac 12 \right)  \sum_\xi \xi^{-n_k} \alpha_{2s}(r_k,g_{\lambda,\xi} \cdot \omega_{\xi,j-s};z_0/k).
\end{equation}

  We next extend this formula  to all $N$ by the procedure used in Section \ref{boh}. Similarly to \e{pal}, the correction factor is
\begin{equation*}
  e^{\lambda v z/N}\prod_{r=0}^{v-1} \left(1-\xi^{-r} e^{ z} e^{- r z/N} \right) =  e^{\lambda v z/N} \sum_{n=0}^\infty \frac{\varphi_{\xi,n}(z,v)}{N^n} = \sum_{n=0}^\infty \frac{\varphi_{\lambda,\xi,n}(z,v)}{N^n}
\end{equation*}
for $0\lqs v\lqs k-1$, $\varphi_{\xi,n}(z,v)$ given explicitly in \e{dap} and
\begin{equation*}
  \varphi_{\lambda,\xi,n}(z,v) := \sum_{j=0}^n \frac{(\lambda v z)^j}{j!} \varphi_{\xi,n-j}(z,v).
\end{equation*}
Next set
\begin{equation*}
  \omega^*_{\lambda,\xi,j}(z,v):= \sum_{n=0}^j \omega_{\xi,n}(z) \cdot \varphi_{\lambda,\xi,j-n}(z,v)
\end{equation*}
and replace $\omega$ by $\omega^*$ in \e{lio} to define
\begin{equation*}
   a^*_{\lambda,j}(N_k,n_k) := \frac{2k^{1/2}}{\pi i }\sum_{s=0}^{j} \G\left(s+\frac 12 \right)
 \sum_\xi \xi^{-n_k}
 \alpha_{2s}(r_k,g_{\lambda,\xi} \cdot \omega^*_{\lambda,\xi,j-s}(\cdot,v);z_0/k)
\end{equation*}
for $N_k \equiv N \equiv -v \bmod k$.
Lastly
\begin{equation} \label{wbuzz}
  a_{\lambda,j}(N_k,n_k) := w_0^{-v/k} \sum_{n=0}^j v^{j-n} \binom{-2-n}{j-n} a^*_{\lambda,n}(N_k,n_k)
\end{equation}
gives the desired formula and this completes the proof of Theorem \ref{ma2}.
\end{proof}
With a similar computation to that of \e{cas2} at the end of Section \ref{boh},  the coefficient for the main term is
\begin{multline} \label{wcas2}
  a_{\lambda,0}(N_k,n_k) =    \frac{z_0}{\pi i} e^{- z_0(\lambda/k+1/2)}
  \frac{(w_0 /k)^{1/2}}{(1-e^{ z_0/k})^{1/2} } \\
\times \sum_\xi \bigg[\xi^{-n_k} \prod_{j=1}^{k-1} \left( \frac{1-\xi^{-j}e^{ z_0/k}}{1-\xi^{-j}}\right)^{j/k-1/2}
\times w_0^{N_k/k}\prod_{j=1}^{N_k} \left( 1-\xi^{j} e^{ z_0/k}\right)^{-1}\bigg].
\end{multline}
The following tables give examples of the approximations of Theorem \ref{ma2} to the first, second and fourth waves, with $\lambda$ equalling $10/7$, $2$ and $3/4$ respectively. The exact values of the waves are computed with \cite[Thm. 2.4]{OS18b}, which is very similar to Proposition \ref{kop2}.

\begin{table}[ht]
\centering
\begin{tabular}{ccc}
\hline
 $r$   & Theorem \ref{ma2} & \\
\hline
 $1$   &  $\textcolor{gray}{-5.4037745492500079} \times 10^{96}$ & \\
 $3$    & $-3.677\textcolor{gray}{9982882192229} \times 10^{96}$ & \\
 $5$   &  $-3.67756\textcolor{gray}{17167251202} \times 10^{96}$ & \\
 $7$    & $-3.677562198\textcolor{gray}{7899526} \times 10^{96}$  &  \\
\hline
\phantom{$W_{1}(3500,5000)$} & $-3.6775621984857302 \times 10^{96}$ & $W_{1}(3500,5000)$\\
\hline
\end{tabular}
\caption{The approximations of Theorem \ref{ma2} to $W_{1}(3500,5000)$.}
\end{table}

\begin{table}[ht]
\centering
\begin{tabular}{ccc}
\hline
 $r$   & Theorem \ref{ma2} & \\
\hline
 $1$   &  $1.2\textcolor{gray}{801787698217348} \times 10^{53}$ & \\
 $3$    & $1.242\textcolor{gray}{3916766083540} \times 10^{53}$ & \\
 $5$   &  $1.2424007\textcolor{gray}{469056981} \times 10^{53}$ & \\
 $7$    & $1.2424007\textcolor{gray}{533841407} \times 10^{53}$  &  \\
\hline
\phantom{$W_{2}(4001,8002)$} & $1.2424007618319874 \times 10^{53}$ & $W_{2}(4001,8002)$\\
\hline
\end{tabular}
\caption{The approximations of Theorem \ref{ma2} to $W_{2}(4001,8002)$.} \label{tabwv2}
\end{table}

\begin{table}[ht]
\centering
\begin{tabular}{ccc}
\hline
 $r$   & Theorem \ref{ma2} & \\
\hline
 $1$   &  $-1.1\textcolor{gray}{915023894770854} \times 10^{23}$ & \\
 $3$    & $-1.18891\textcolor{gray}{47720679222} \times 10^{23}$ & \\
 $5$   &  $-1.18891888\textcolor{gray}{77091459} \times 10^{23}$ & \\
 $7$    & $-1.1889188816\textcolor{gray}{772328} \times 10^{23}$  &  \\
\hline
\phantom{$W_{4}(4000,3000)$} & $-1.1889188816869245 \times 10^{23}$ & $W_{4}(4000,3000)$\\
\hline
\end{tabular}
\caption{The approximations of Theorem \ref{ma2} to $W_{4}(4000,3000)$.} \label{tabwv}
\end{table}



{\small
\bibliography{part-bib}

\begin{thebibliography}{FGD95}

\bibitem[Ahl78]{Al}
Lars~V. Ahlfors.
\newblock {\em Complex analysis}.
\newblock McGraw-Hill Book Co., New York, third edition, 1978.
\newblock An introduction to the theory of analytic functions of one complex
  variable, International Series in Pure and Applied Mathematics.

\bibitem[And03]{An}
George~E. Andrews.
\newblock Partitions: at the interface of {$q$}-series and modular forms.
\newblock {\em Ramanujan J.}, 7(1-3):385--400, 2003.
\newblock Rankin memorial issues.

\bibitem[Apo76]{ApIntro}
Tom~M. Apostol.
\newblock {\em Introduction to analytic number theory}.
\newblock Springer-Verlag, New York-Heidelberg, 1976.
\newblock Undergraduate Texts in Mathematics.

\bibitem[DG14]{DrGe}
Michael Drmota and Stefan Gerhold.
\newblock Disproof of a conjecture by {R}ademacher on partial fractions.
\newblock {\em Proc. Amer. Math. Soc. Ser. B}, 1:121--134, 2014.

\bibitem[Dic66]{Di}
Leonard~Eugene Dickson.
\newblock {\em History of the theory of numbers. {V}ol. {II}: {D}iophantine
  analysis}.
\newblock Chelsea Publishing Co., New York, 1966.

\bibitem[FGD95]{Fl95}
Philippe Flajolet, Xavier Gourdon, and Philippe Dumas.
\newblock Mellin transforms and asymptotics: harmonic sums.
\newblock {\em Theoret. Comput. Sci.}, 144(1-2):3--58, 1995.
\newblock Special volume on mathematical analysis of algorithms.

\bibitem[FOR13]{FOR}
Amanda Folsom, Ken Ono, and Robert~C. Rhoades.
\newblock Mock theta functions and quantum modular forms.
\newblock {\em Forum Math. Pi}, 1:e2, 27, 2013.

\bibitem[GZ18]{GZ}
Stavros Garoufalidis and Don Zagier.
\newblock Asymptotics of {N}ahm sums at roots of unity.
\newblock {\em arXiv:1812.07690}, 2018.

\bibitem[Hik03]{Hi}
Kazuhiro Hikami.
\newblock Volume conjecture and asymptotic expansion of {$q$}-series.
\newblock {\em Experiment. Math.}, 12(3):319--337, 2003.

\bibitem[O'S15]{OS15}
Cormac O'Sullivan.
\newblock On the partial fraction decomposition of the restricted partition
  generating function.
\newblock {\em Forum Math.}, 27(2):735--766, 2015.

\bibitem[O'S16a]{OS16}
Cormac O'Sullivan.
\newblock Asymptotics for the partial fractions of the restricted partition
  generating function {I}.
\newblock {\em Int. J. Number Theory}, 12(6):1421--1474, 2016.

\bibitem[O'S16b]{OS16b}
Cormac O'Sullivan.
\newblock Asymptotics for the partial fractions of the restricted partition
  generating function {II}.
\newblock {\em Integers}, 16:Paper No. A78, 73 pp, 2016.

\bibitem[O'S16c]{OS16c}
Cormac O'Sullivan.
\newblock Zeros of the dilogarithm.
\newblock {\em Math. Comp.}, 85(302):2967--2993, 2016.

\bibitem[O'S18]{OS18b}
Cormac O'Sullivan.
\newblock Partitions and {S}ylvester waves.
\newblock {\em Ramanujan J.}, 47(2):339--381, 2018.

\bibitem[O'S19]{OSper}
Cormac O'Sullivan.
\newblock Revisiting the saddle-point method of {P}erron.
\newblock {\em Pacific J. Math.}, 298(1):157--199, 2019.

\bibitem[Per17]{Pe17}
Oskar Perron.
\newblock \"{U}ber die n\"{a}herungsweise {B}erechnung von {F}unktionen
  gro{\ss}er {Z}ahlen.
\newblock {\em Sitzungsber. Bayr. Akad. Wissensch. (M\"unch. Ber.)}, pages
  191--219, 1917.

\bibitem[Rad37]{Ra2}
Hans Rademacher.
\newblock A convergent series for the partition function p(n).
\newblock {\em Proc. Natl. Acad. Sci. USA}, 23(2):78--84, 1937.

\bibitem[Rad73]{Ra}
Hans Rademacher.
\newblock {\em Topics in analytic number theory}.
\newblock Springer-Verlag, New York, 1973.
\newblock Edited by E. Grosswald, J. Lehner and M. Newman, Die Grundlehren der
  mathematischen Wissenschaften, Band 169.

\bibitem[ST88]{SrTo88}
H.~M. Srivastava and Pavel~G. Todorov.
\newblock An explicit formula for the generalized {B}ernoulli polynomials.
\newblock {\em J. Math. Anal. Appl.}, 130(2):509--513, 1988.

\bibitem[SZ13]{SZ}
Andrew~V. Sills and Doron Zeilberger.
\newblock Rademacher's infinite partial fraction conjecture is (almost
  certainly) false.
\newblock {\em J. Difference Equ. Appl.}, 19(4):680--689, 2013.

\bibitem[Zag01]{Za01}
Don Zagier.
\newblock Vassiliev invariants and a strange identity related to the {D}edekind
  eta-function.
\newblock {\em Topology}, 40(5):945--960, 2001.

\bibitem[Zag07]{Zag07}
Don Zagier.
\newblock The dilogarithm function.
\newblock In {\em Frontiers in number theory, physics, and geometry. {II}},
  pages 3--65. Springer, Berlin, 2007.

\bibitem[Zud19]{Zu}
Wadim Zudilin.
\newblock Congruences for $q$-binomial coefficients.
\newblock {\em arXiv:1901.07843}, 2019.

\end{thebibliography}
}

{\small 
\vskip 5mm
\noindent
\textsc{Dept. of Math, The CUNY Graduate Center, 365 Fifth Avenue, New York, NY 10016-4309, U.S.A.}

\noindent
{\em E-mail address:} \texttt{cosullivan@gc.cuny.edu}
}

\end{document}